\documentclass[a4paper,11pt,reqno]{amsart}
\usepackage{mathrsfs}
\usepackage{nicefrac}

\usepackage{a4wide}
\usepackage[applemac]{inputenc}
 \usepackage[pdfpagemode=UseNone,bookmarksopen=false,colorlinks=true]{hyperref}
\usepackage{enumerate}
\usepackage[shortlabels]{enumitem}
\usepackage{bbm,bm}

\hypersetup{ colorlinks=true, linkcolor=blue, citecolor=blue,
  filecolor=blue, urlcolor=blue}

\usepackage[comma, sort&compress]{natbib}
\usepackage{graphicx}

\usepackage{amsmath}
\usepackage{amsthm}
\usepackage{amssymb}
\usepackage{color}
\definecolor{mygray}{gray}{0.6}
\usepackage[normalem]{ulem} 
\usepackage {ulem} 
\usepackage[osf,sc]{mathpazo}

\usepackage{enumerate}

\usepackage{url}

\theoremstyle{plain}
\newtheorem{theorem}{Theorem}[section]
\newtheorem{proposition}[theorem]{Proposition}
\newtheorem{lemma}[theorem]{Lemma}

\newtheorem{corollary}[theorem]{Corollary}

\newtheorem{definition}[theorem]{Definition}

\theoremstyle{definition}

\theoremstyle{remark}
\newtheorem{remark}[theorem]{Remark}

\newcommand{\N}{\ensuremath{\mathbb{N}}}
\newcommand{\Z}{\ensuremath{\mathbb{Z}}}

\newcommand{\be}{\begin{equation}}
\newcommand{\ee}{\end{equation}}
\newcommand{\e}{\mathrm e}
\newcommand{\De}{\mathrm d}
\renewcommand{\c}{\mathrm c}

\renewcommand{\c}{\mathrm{c}}

\newcommand{\eq}[1]{\begin{equation#1}}
\newcommand{\eeq}[1]{\end{equation#1}}

\numberwithin{equation}{section}

\makeatletter
\def\paragraph{\@startsection{paragraph}{4}%
  \z@\z@{-\fontdimen2\font}%
  {\normalfont\bfseries}}
\makeatother
\title[Exponential decay for the pinned  membrane model]{Exponential decay of covariances for the supercritical membrane model}

\author[E.Bolthausen]{Erwin Bolthausen}
\address{Institut f\"ur Mathematik, Universit\"at Z\"urich, Winterthurerstrasse 190, CH-8057, Zurich, Switzerland}
\email{eb@math.uzh.ch}

\author[A. Cipriani]{Alessandra Cipriani}
\address{Department of Mathematical Sciences, University of Bath, Claverton Down, Bath, BA2 7AY, United Kingdom.}
\email{A.Cipriani@bath.ac.uk}

\author[N. Kurt]{Noemi Kurt}
\address{Technische Universit\"at Berlin,
MA 766, Strasse des 17. Juni 136, 10623
Berlin, Germany}
\email{kurt@math.tu-berlin.de}

\date{\today}
\begin{document}

\date{}
\begin{abstract}
We consider the membrane model, that is the centered Gaussian field on $\Z^d$ whose covariance matrix is given by the inverse of the discrete Bilaplacian. We impose a $\delta-$pinning condition, giving a reward of strength $\varepsilon$ for the field to be $0$ at any site of the lattice. In this paper we prove that in dimensions $d\geq 5$ covariances of the pinned field decay at least exponentially, as opposed to the field without pinning, where the decay is polynomial. The proof is based on estimates for certain discrete weighted norms, a percolation argument and on a Bernoulli domination result.
\end{abstract}
\maketitle
\section{Introduction}

Effective interface models are well-studied real-valued random fields, defined
for instance on the lattice $\mathbb{Z}^{d}$, which predict the behavior of
polymers and interfaces between two states of matter. The best known examples
are the gradient models $\varphi=\left\{  \varphi_{x}\right\}
_{x\in\mathbb{Z}^{d}}$ which (in formal notation) are of the form%
\[
P\left(  \De \varphi\right)  :=\frac{1}{Z}\exp\left[  -H\left(  \varphi\right)
\right]  \prod_{x}\De\varphi_{x},
\]
with the Hamitonian%
\[
H\left(  \varphi\right)  :=\sum_{x,\,y\in\mathbb{Z}^{d},\ \left\Vert
x-y\right\Vert =1}V\left(  \varphi_{x}-\varphi_{y}\right)
\]
where $V:\mathbb{R\rightarrow R}$ is the interaction function, satisfying
$V\left(  x\right)  \rightarrow\infty$ for $\left\Vert x\right\Vert
\rightarrow+\infty$. The measure has to be defined through a thermodynamic
limit. In the case $V\left(  x\right) : =\beta x^{2},$ the model is Gaussian,
but it is defined on the whole of $\mathbb{Z}^{d}$ only for $d\geq3.$ For
lower dimensions, one has to restrict $x$ to a finite set, and put boundary
conditions. This is the so-called Gaussian free field which has attracted tremendous
attention recently for $d=2.$ One simplifying feature of the free field is
that the covariances of the model are given in terms of the Green's function
of a standard random walk on the lattice, and many properties of the field can
be derived from properties of the random walk. This has led to powerful
techniques for analysing the model. The case where $V$ is not quadratic is
much more complicated. If $V$ is convex, there is still a random walk
representation of the correlation, the Helffer-Sj\"{o}strand representation,
but in the case of non-convex $V$, random walk techniques cannot be applied,
and many of the very basic questions are still open. For a recent
investigation, see \cite{Adams, AKM}.

The so-called massive free field has the Hamiltonian%
\[
H\left(  \varphi\right)  :=\beta\sum_{x,\,y,\ \left\Vert x-y\right\Vert
=1}\left(  \varphi_{x}-\varphi_{y}\right)  ^{2}+m\sum_{x}\varphi_{x}%
^{2},\quad \beta,\,m>0,
\]
and it is a Gaussian field which
is well-defined on the full lattice in any dimension, and has exponentially
decaying covariances. This just comes from the fact that the covariances are
given by the Green's function of a random walk with a positive killing rate \cite[Theorem~8.46]{veleniknotes}.

It is quite astonishing that an exponential decay of correlations, in physics
jargon a positive mass, also appears when the free field Hamiltonian is
perturbed by an arbitrary small attraction to the origin, for instance in the
form%
\begin{equation}
H\left(  \varphi\right)  :=\beta\sum_{x,\,y,\ \left\Vert x-y\right\Vert
=1}\left(  \varphi_{x}-\varphi_{y}\right)  ^{2}+a\sum_{x}\mathbf{1}_{\left[
-b,b\right]  }\left(  \varphi_{x}\right)  \label{window_pinning}%
\end{equation}
with $a,\,b>0$ (see \citet[Section 5]{velenikloc}). A somewhat simpler case is that of so-called $\delta
$-pinning where the reference measure $\prod_{x}\De\varphi_{x}$ is replaced by
$\prod_{x}\left(  \De\varphi_{x}+\varepsilon\delta_{0}\left(  \De\varphi
_{x}\right)  \right)  ,$ and which can be obtained from (\ref{window_pinning})
by a suitable limiting procedure letting $b\rightarrow0,\ a\rightarrow+\infty.$ All the
proofs we are aware of rely heavily on random walk representations.

Our main object here is to discuss similar properties for the $\delta$-pinned
membrane model which has the Hamiltonian%
\[
H\left(  \varphi\right)  :=\frac{1}{2}\sum_{x\in \Z^d}\left(  \Delta\varphi_{x}\right)  ^{2}%
\]
where $\Delta$ is the discrete Laplace operator on functions $f:\mathbb{Z}%
^{d}\rightarrow\mathbb{R}$, defined by%
\eq{}\label{def:lapl}
\Delta f\left(  x\right)  :=\frac{1}{2d}\sum_{y:\,\left\Vert y-x\right\Vert
=1}\left(  f\left(  y\right)  -f\left(  x\right)  \right)  .
\eeq{}
We leave out the temperature parameter $\beta$ as it just leads to a trivial
rescaling of the field.

While the free field (described for example in \citet[Chapter 8]{veleniknotes}) is used to model polymers or
interfaces with a tendency to maintain a constant mean height, the the
membrane model appears in physical and biological research to shape interfaces
that tend to have constant curvature \citep{Hiergeist,Leibler,Lipowsky,Ruiz}. In solid state physics one often
considers models with mixed gradient and Laplacian Hamiltonian, but we will
not discuss such cases here. The two models share many common characteristics,
for instance their variances are uniformly bounded in $\mathbb{Z}^{d}$ if the
dimension is large enough, that is $d\geq3$ for the gradient case resp.
$d\geq5$ for the membrane model, and have variances growing logarithmically in
$d=2$ resp. $d=4$.

The main topic of the present paper is an investigation of the decay of
correlations for the membrane model in dimensions $d\geq5.$ We restrict to the
case of $\delta$-pinning for technical reasons. We prove that the field
becomes \textquotedblleft massive\textquotedblright, i.e. has exponentially
decaying correlation for any positive pinning parameter.

The main difficulty when compared with the proofs of similar results for the
free field is the absence of useful random walk representations for the
covariances and correlation inequalities. Random walk representations for
gradient fields have been very important since the celebrated work \cite{BFS}.
There is a variant of a random walk representation in the case of the membrane
model \cite[Section 2]{Kurt_d4}, but only in the presence of particular boundary conditions, or in the
case of the field on the whole lattice in the absence of boundary conditions. Results on the membrane model with pinning were shown in $(1+1)$ dimensions by \cite{CaravennaDeuschel_pin} using a renewal type of argument which, however, is not applicable in higher dimensions. We would like to mention also the work \cite{AKW} on large deviation principles under a Laplacian interaction without using renewal type arguments.

\noindent\textbf{Structure of the paper.} The structure of the paper is as
follows: in Section \ref{sec:model} we give precise definitions on the
membrane model and the statement of our main theorem. We recall general
results, including Bernoulli domination, in Section~\ref{sec:general}. In
Section~\ref{sec:main} we prove our main theorem.

\noindent\textbf{Acknowledgements.} We are grateful to Vladimir Maz'ya who
gave a significant input to the present work by showing how to prove the
exponential decay for the Bilaplacian in the continuous space with a
sufficiently dense set of deterministic \textquotedblleft
traps\textquotedblright\ and appropriate boundary conditions. For more information on the analytic background the reader can consult \citet{Mazya}.

This work was performed in part during visits of the first author to the TU
Berlin and WIAS Berlin, and of the last two authors to the University of
Zurich. We thank these institutions for their hospitality. Francesco
Caravenna, Jean-Dominique Deuschel and Rajat Subhra Hazra are acknowledged for
feedback and helpful discussions.

\section{The model and main results\label{sec:model}}

\subsection{Basic notations}

We will work on the $d$-dimensional integer lattice $\mathbb{Z}^{d}$, and in
the present paper our focus will be in $d\geq5$, although the basic definition
is well-posed in all dimensions. Also, some of the partial results which don't
rely on the dimension restriction will be stated and proved in generality.

For $N\in\mathbb{N},$ let $V_{N}:=[-\nicefrac{N}{2},\,\nicefrac{N}{2}]^{d}\cap\mathbb{Z}^{d}$ and
$V_{N}^{\mathrm{c}}:=\mathbb{Z}^{d}\setminus V_{N}$.

For $x,\,y\in\mathbb{Z}^{d},$ \textrm{d}$\left(  x,\,y\right)  $ is the graph
distance between $x$ and $y$ on the lattice with nearest-neighbor bonds, i.e.
the $\ell_{1}$-norm of $x-y.$ With $\left\Vert \cdot\right\Vert ,$ we denote the
Euclidean norm.

We will use $L$ as a generic positive constant which depends only on the
dimension $d$, not necessarily the same at different occurencies, and also not
necessarily the same within the same formula. The dependence on $d$ will not
be mentioned, but dependence on other parameters will be noted by writing
$L\left(  k\right)  $ or $L\left(  \varepsilon\right)  $, for instance.

We will consider real valued random fields $\left\{  \varphi_{x}\right\}
_{x\in\mathbb{Z}^{d}}.$ For $A\subset\mathbb{Z}^{d},$ we write $\mathcal{F}%
_{A}$ for the $\sigma$-field generated by the random variables $\left\{\varphi
_{x},\,x\in A\right\}.$ To be definite, we can of course have all the measures
constructed on $\mathbb{R}^{\mathbb{Z}^{d}},$ equipped with the product
$\sigma$-field.

We will typically use $x,\,y$ for points in $\mathbb{Z}^{d}.$ If we write
$\sum_{x}$, this means summation over all $\mathbb{Z}^{d}.$ We will use $e$
exclusively for the $2d$ elements of $\mathbb{Z}^{d}$ which are neighbors of
$0.$ To keep notations less heavy, $\sum_{e}$ means that we sum over all these elements, and similarly for other discrete differential operators we will introduce. For a function $f$
on $\mathbb{Z}^{d}$, we write%
\[
D_{e}f\left(  x\right)  :=f\left(  x+e\right)  -f\left(  x\right)  .
\]
We write $\nabla f$ for the vector $\left(  D_{e}f\right)  _{e},$ and
$\nabla^{2}f$ for the matrix $\left(  D_{e}D_{e^{\prime}}f\right)
_{e,\,e^{\prime}}$, and similarly for the higher order derivatives which are
denoted by $\nabla^{3},\,\nabla^{4}$ etc. Remark that $\nabla^{k}f\left(  x\right)
$ depends on all the values $f\left(  y\right)  $ with $\mathrm{d}\left(  y,\,x\right)
\leq k.$ We write%
\[
\left\Vert \nabla^{k}f\left(  x\right)  \right\Vert ^{2}=\left\Vert \nabla^{k}f\left(  x\right)  \right\Vert ^{2}_{2}:=\sum_{e_{1}%
,\ldots,e_{k}}\left\vert D_{e_{1}}D_{e_{2}}\cdots D_{e_{k}}f\left(  x\right)
\right\vert^{2}.
\]
We also define $\left\Vert \nabla^{k}f\left(  x\right)  \right\Vert_{\infty}:=\sup_{e_1,\,\ldots,\,e_k}\left|D_{e_1}\cdots D_{e_k}f(x)\right|$.
The Laplacian in \eqref{def:lapl} can be rewritten as%
\[
\Delta f\left(  x\right)  :=\frac{1}{2d}\sum_{e}D_{e}f\left(  x\right)  .
\]
Remark that although the right hand side looks like being a first order
discrete derivative, it is of course a second order one through the presence
of $e$ and $-e$ in the summation. Namely, if we define only the positive coordinate directions as $\{e^{(1)},\,\ldots,\,e^{(d)}\}$, then the alternative definition
\eq{}\label{eq:char_lapl}
\Delta f(x)=-\frac{1}{2d}\sum_{i=1}^d D_{e^{(i)}}D_{-e^{(i)}}f(x)
\eeq{}
holds.
For two square summable functions $f,g$ on $\mathbb{Z}^{d},$ we write%
\[
\left\langle f,g\right\rangle :=\sum_{x\in\mathbb{Z}^{d}}f\left(  x\right)
g\left(  x\right)  .
\]
Summation by parts leads to the following properties:

\begin{lemma}
\label{lem:sumbyparts} Let $f,g:\mathbb{Z}^{d}\rightarrow\mathbb{R}$ be square
summable functions.

\begin{enumerate}
\item[a)] For any $e$%
\begin{equation}
\left\langle D_{e}f,g\right\rangle =\left\langle f,D_{-e}g\right\rangle .
\label{part1}%
\end{equation}

\item[b)]
\begin{equation}
\left\langle \Delta f,g\right\rangle =\left\langle f,\Delta g\right\rangle.
\label{part2}%
\end{equation}

\item[c)]
\begin{equation}
\sum_{e}\left\langle D_{e}f,D_{e}g\right\rangle =-4d\left\langle f,\Delta
g\right\rangle . \label{part4}%
\end{equation}

\end{enumerate}
\end{lemma}

\subsection{The membrane model and statement of the main result}

\begin{definition}[\cite{Sakagawa}, \cite{velenikloc}, \cite{Kurt_thesis}] Let
$W\neq\emptyset$ be a finite subset of $\mathbb{Z}^{d}.$ The \emph{membrane
model} on $W$ is the random field $\{\varphi_{x}\}_{x\in\mathbb{Z}^{d}}%
\in\mathbb{R}^{\mathbb{Z}^{d}}$ with zero boundary conditions outside $W$,
whose distribution is given by
\begin{equation}
P_{W}(\mathrm{d}\varphi)=\frac{1}{Z_{W}}\exp\left(  -\frac{1}{2}\left\langle
\Delta\varphi,\Delta\varphi\right\rangle \right)  \prod_{x\in W}%
\mathrm{d}\varphi_{x}\prod_{x\in W^{\c}}\delta_{0}(\mathrm{d}\varphi_{x}),
\label{def:field}%
\end{equation}
where $Z_{W}$ is the normalizing constant.

In the case $W:=V_{N},$ we simply write $P_{N}$ instead of $P_{V_{N}}.$
\end{definition}

It is notationally convenient to define the field $\left\{  \varphi
_{x}\right\}  $ for $x\in\mathbb{Z}^{d}$, but as $\varphi_{x}=0$ for $x\notin
W,$ it is just a centered Gaussian random vector $\left\{  \varphi
_{x}\right\}  _{x\in W}$. By (\ref{part2}), one has%
\[
\left\langle \Delta\varphi,\Delta\varphi\right\rangle =\left\langle
\varphi,\Delta^{2}\varphi\right\rangle .
\]
Remark that in the inner product on the left hand side, one cannot restrict
the sum to $W$ even if $\varphi$ is $0$ outside $W.$ There is in fact a
contribution from the points at distance $1$ to $W.$ In contrast, in the inner
product on the right hand side, the sum is only over $W$. $P_{W}$, when
regarded as a law of a $\mathbb{R}^{W}$-valued vector, has density
proportional to%
\[
\exp\left(  -\frac{1}{2}\left\langle \varphi,\Delta_{W}^{2}\varphi
\right\rangle \right)
\]
where $\Delta_{W}^{2}=\big(\Delta^{2}(x,\,y)\big)_{\{x,\,y\in W\}}$ is the the
restriction of the Bilaplacian to $W$. Actually, in order to make
(\ref{def:field}) meaningful, one needs that $\Delta_{W}^{2}$ is positive
definite. This follows from the maximum principle for $\Delta$. In fact
$\left\langle \Delta\varphi,\Delta\varphi\right\rangle >0$ holds for all
$\varphi$ which do not vanish identically, and are $0$ on $W^{\c}$. This proves
the positive definiteness of $\Delta_{W}^{2}$.

The covariances of the membrane model are given as%
\begin{equation}
G_{W}(x,\,y):=\operatorname*{cov}\nolimits_{P_{W}}(\varphi_{x},\varphi
_{y})=\left(  \Delta_{W}^{2}\right)  ^{-1}(x,\,y),\quad x,\,y\in W,
\label{Grreen_by_Bilaplacian}%
\end{equation}

It is convenient to extend $G_{W}$ to $x,\,y\in\mathbb{Z}^{d}$ by setting the
entries to $0$ outside $W\times W.$ For $x\in W,$ the function $\mathbb{Z}%
^{d}\ni y\mapsto G_{W}\left(  x,\,y\right)  $ is the unique solution of the
boundary value problem \citep{Kurt_d4}
\begin{equation*}
\left\{
\begin{array}
[c]{ll}%
\Delta^{2}G_{W}(x,\,y)=\delta_{x,\,y}, & y\in W\\
G_{W}(x,\,y)=0, & y\notin W
\end{array}
\right.  .
\end{equation*}
For $d\geq5$ the weak limit $P:=\lim\nolimits_{N\rightarrow\infty}P_{N}$ exits
\cite[Section II]{Sakagawa}. Under $P$, the canonical coordinates $\left\{
\varphi_{x}\right\}  _{x\in\mathbb{Z}^{d}}$ form a centered Gaussian random
field with covariance given by
\begin{equation*}
G(x,\,y)=\Delta^{-2}(x,\,y)=\sum_{z\in\mathbb{Z}^{d}}\Delta^{-1}(x,z)\Delta
^{-1}(z,y)=\sum_{z\in\mathbb{Z}^{d}}\Gamma(x,z)\Gamma(z,y),
\end{equation*}
where $\Gamma$ is the Green's function of the discrete Laplacian on
$\mathbb{Z}^{d}$. In particular observe that \eq{}\label{eq:GtoGamma}G(0,\,0)<+\infty.
\eeq{}
The matrix $\Gamma$ has a representation in terms of the
simple random walk $(S_{m})_{m\geq0}$ on $\mathbb{Z}^{d}$ given by
\[
\Gamma(x,\,y)=\sum_{m\geq0}\mathrm{P}_{x}[S_{m}=y]
\]
($\mathrm{P}_{x}$ is the law of $S$ starting at $x$). This entails that
\begin{equation*}
G(x,\,y)=\sum_{m\geq0}(m+1)\mathrm{P}_{x}[S_{m}=y]=\mathrm{E}_{x,\,y}\left[
\sum_{\ell,\,m=0}^{\infty}\mathbf{1}_{\left\{  S_{m}=\tilde{S}_{\ell}\right\}
}\right]  \label{eq:RW-repG}%
\end{equation*}
where $S$ and $\tilde{S}$ are two independent simple random walks starting at
$x$ and $y$ respectively. $\Gamma$ and $G$ are translation invariant. Using
the above representation one can easily derive the following property of
the covariance:

\begin{lemma}[{\citet[Lemma 5.1]{Sakagawa}}]
\label{lemma: covariance:mm} For $d\geq5$ there exists a constant $\kappa
_{d}>0$%
\begin{equation*}
\lim_{\Vert x\Vert\rightarrow\infty}\frac{G(0,x)}{\Vert x\Vert^{4-d}}%
=\kappa_{d} 
\end{equation*}

\end{lemma}

In other words, as $\Vert x-y\Vert\rightarrow+\infty,$ the covariance between
$\varphi_{x}$ and $\varphi_{y}$ decays like $\kappa_{d}\Vert x-y\Vert^{4-d}$
in the supercritical dimensions.

For $d=4$, $\lim_{N\rightarrow+\infty}P_{N}$ does not exist, and in fact,
$\operatorname*{var}_{P_{N}}\left(  \varphi_{0}\right)  \rightarrow+\infty.$ It
is known that $G_{N}(x,\,y)$ behaves in first order as $\gamma_{4}(\log
N-\log\Vert x-y\Vert)$ for some $\gamma_{4}\in(0,\,+\infty),$ if $x$ and $y$
are not too close to the boundary of $V_{N},$ see \citet[Lemma~2.1]{Cip13}.

\begin{definition}
[Pinned membrane model]Let $\varepsilon>0$. The \emph{membrane model on }%
$W$\emph{ with pinning of strength $\varepsilon$} is defined as
\begin{equation}
P_{W}^{\varepsilon}(\mathrm{d}\varphi)=\frac{1}{Z_{W}^{\varepsilon}}%
\exp\left(  -\frac{1}{2}\left\langle \Delta\varphi,\Delta\varphi\right\rangle
\right)  \prod_{x\in W}\left(  \mathrm{d}\varphi_{x}+\varepsilon\delta
_{0}(\mathrm{d}\varphi_{x})\right)  \prod_{x\in W^{\c}}\delta_{0}%
(\mathrm{d}\varphi_{x}),\label{def:pinned}%
\end{equation}
where $Z_{W}^{\varepsilon}$ is the normalizing constant%
\[
Z_{W}^{\varepsilon}:=\int\exp\left(  -\frac{1}{2}\left\langle \Delta
\varphi,\Delta\varphi\right\rangle \right)  \prod_{x\in W}\left(
\mathrm{d}\varphi_{x}+\varepsilon\delta_{0}(\mathrm{d}\varphi_{x})\right)
\prod_{x\in W^{\c}}\delta_{0}(\mathrm{d}\varphi_{x}).
\]
In case $W=V_{N}$, we write $P_{N}^{\varepsilon}$ and $Z_{N}^{\varepsilon}$ instead.
\end{definition}

Our main result shows that for any positive pinning strength $\varepsilon$ the
correlations between two points decay exponentially in the distance. 

\begin{theorem}
[Decay of covariances, supercritical case]\label{Thm_main} Let $d\geq5$ and
$\varepsilon>0$. Then there exist $C,\eta>0$ depending on $\varepsilon$ and
$d$, but not on $N,$ such that
\begin{equation*}
\left\vert E_{N}^{\varepsilon}[\varphi_{x}\varphi_{y}]\right\vert \leq
C\mathrm{e}^{-{\eta}\left\Vert x-y\right\Vert } 
\end{equation*}
whenever $x,\,y\in V_{N}$.
\end{theorem}
\begin{remark}{Note that one can show that adding a mass to the membrane model implies exponential decay of correlations}.
\end{remark}
{\subsection{Proof outline}
To motivate our approach, consider the following PDE problem in continuous
space. Let%
\[
\Omega:=\mathbb{R}^{n}\backslash\bigcup\nolimits_{i}B_{r}\left(  x_{i}\right)
,
\]
where $\left\{  B_{r}\left(  x_{i}\right)  \right\}_i  $ is a collection of
closed non-overlapping balls of radius $r$ which is sufficiently dense. For instance, assume
to take $\left\{  x_{i}\right\} _i :=\mathbb{Z}^{d},$ and
$r\leq1/4$. The function $u:\Omega\rightarrow\mathbb{R}$ is assumed to be smooth and to satisfy $\Delta^{2}u=f$, where $f$ is a smooth function on $\Omega$ of
compact support, $\Delta^{2}$ is the continuum bilaplacian, and $u$ and
$\nabla u$ have $0$-boundary conditions at $\partial\Omega$. Is it true that $u$ is
exponentially decaying at infinity, assuming only some mild growth condition?
One can answer positively to this question as follows (the authors learned this argument from Vladimir Maz'ya): the key observation is
that if $u$ on $\Omega$ satisfies $0$-boundary conditions, one can obtain the
equivalence of the standard second order Sobolev norm $\left\Vert u\right\Vert
_{H^{2}\left(  \Omega\right)  }$ with the $L_{2}$-norm of the second
derivative; in other words, the $L_{2}$-norm of $u$ and of $\nabla u$ can be
estimated by the $L_{2}$-norm of the second derivatives. In our case, this
follows by selecting a linear path from every point $x\in\Omega$ to the
boundary $\partial\Omega,$ and then using the $0$-boundary conditions and
partial integration along the path to estimate $f$ and $\nabla f$ in terms of
the second derivative. Such equivalences are discussed in much greater
generality in \citet{MazyaBook}.
Choose now a sequence of concentric balls $C_{n}:=B_{n}\left(  0\right)
\cap\Omega,$ starting with $n$ such that $C_{n}$ contains the support of $f,$
and choose smooth functions $\eta_{n}:\mathbb{\Omega\rightarrow}\left[
0,1\right]  $, interpolating between $1$ outside $C_{n+1}$ and $0$ on $C_{n}.$
Then%
\begin{align*}
\left\Vert u\right\Vert _{H^{2}\left(  C_{n+1}^{c}\right)  }  & =\left\Vert
\eta_{n}u\right\Vert _{H^{2}\left(  C_{n+1}^{c}\right)  }\leq\left\Vert
\eta_{n}u\right\Vert _{H^{2}\left(  \Omega\right)  }\\
& \leq\operatorname*{const}\times\left\Vert \nabla^{2}\left(  \eta
_{n}u\right)  \right\Vert _{L^{2}\left(  \Omega\right)  }%
=\operatorname*{const}\times\left\Vert \nabla^{2}\left(  \eta_{n}u\right)
\right\Vert _{L^{2}\left(  C_{n+1}\backslash C_{n}\right)  }\\
& \leq\operatorname*{const}\times\left\Vert u\right\Vert _{H^{2}\left(
C_{n+1}\backslash C_{n}\right)  }=\operatorname*{const}\times\left[
\left\Vert u\right\Vert _{H^{2}\left(  C_{n}^{c}\right)  }-\left\Vert
u\right\Vert _{H^{2}\left(  C_{n+1}^{c}\right)  }\right]  ,
\end{align*}
which proves the exponential decay of the Sobolev norms. In the second
inequality, we have used the equivalence of the norms. In the second line,
we have used that $\eta_{n}=1$ outside $C_{n+1}$ and that $u$ is biharmonic. Of
course, we have always assumed as an input that the above Sobolev norms are
finite, but in our problem this will not be a difficulty.
}

{
The application to our setting requires a number of modifications. The first step
is to notice that the environment $\mathcal{A}$ of pinned points,
corresponding to the \textquotedblleft holes\textquotedblright\ $B_{r}\left(
x_{i}\right)  $ above, can be dominated stochastically by a Bernoulli site
percolation measure. The boundary conditions for the discrete derivative on
the pinned sites $\mathcal{A}$ is however not easily computable, and in general it is
not $0.$ For that reason we work with the inner points $\widehat{\mathcal{A}%
},$ and use the fact that the law of this set is dominated by a Bernoulli measure, too.
However the key difficulty is that there is certainly not an upper bound to the distance 
between any point in $\mathbb{Z}^{d}$ to any trapping points in
$\widehat{\mathcal{A}},$ in contrast to the continuum situation
sketched above, and therefore there is no equivalence of norms (with
discrete derivatives, of course). The way to solve this problem is to
introduce random Sobolev norms which depend on the random set $\widehat
{\mathcal{A}},$ and then prove that, in an appropriate sense, the Sobolev norm
involving randomly weighted discrete derivatives up to the second order is
equivalend to one coming from the second derivative only. This however makes
it necessary to adapt the choice of the the sequence $C_{n}$ to the random
trapping set $\mathcal{A}$. Indeed, it is necessary to choose the
interpolating functions $\eta_{n}$ in such a way that the derivatives are
small in regions where there are few points in $\widehat{\mathcal A}$. This leads to a random choice of
the $C_{n}$'s, and in the end, one has to use a percolation argument to prove
that the radius of the $C_{n}$'s still grows linearly in
$n$ with overwhelming probability. This would not be possible choosing the $C_{n}$'s as concentric balls.}
\begin{remark}
A more natural statement would be that $P^{\varepsilon}:=\lim_{N\rightarrow
\infty}P_{N}^{\varepsilon}$ has exponentially decaying covariances.
Unfortunately, we do not know if this limit exists. The proof in \citet{BoltVel} of
the existence of the weak limit in the gradient case uses correlation
inequalities which are not valid in the membrane case.
\end{remark}

\begin{remark}[Outlook on the case $d=4$]
The restriction to $d\geq5$ is coming from a domination of the measure
$\nu_{N}^{\varepsilon}$ defined in (\ref{eq:def_zeta}) from below by a
Bernoulli measure which is true in a strong sense only for $d\geq5.$ The other
steps of the proof do not depend on this dimension restriction in an essential
way. The method we apply here would give for $d=4$ an estimate of $\left\vert
E_{N}^{\varepsilon}[\varphi_{x}\varphi_{y}]\right\vert $ in the form
$\exp\left[  -\eta\left\Vert x-y\right\Vert \left(  \log N\right)  ^{-\alpha
}\right]  $ with some $\eta,\,\alpha>0.$ This is of course disappointing as
for fixed $x,\,y,$ one would not get decay properties which are uniformly in $N$,
and one would also not get boundedness of the variances $\operatorname*{var}%
_{P_{N}^{\varepsilon}}\left(  \varphi_{0}\right)  $. We remark also that with techniques similar to those of the present paper \citet{BCK} show stretched exponential decay of covariances in $d\ge 4$. 

We however expect that with some weaker domination properties, as the one used
in \citet{BoltVel} for $d=2$, one could prove exponential decay also for the
membrane model in $d=4.$ However, the proofs used in \citet{BoltVel} rely again on
correlation inequalities, so a proof eludes us.

It is well possible that exponential decay of correlations is true also for
lower dimensions $d=2,3$, but we do not know of a method which could
successfully be applied.
\end{remark}

\section{General results on the membrane model\label{sec:general}}

Let $B\subset A\Subset\mathbb{Z}^{d}.$ As the Hamiltonian of the
membrane model is represented through an interaction of range $2,$ the
conditional distribution of $\left\{  \varphi_{x}\right\}  _{x\in B}$ under
$P_{A}$ given $\mathcal{F}_{A\backslash B}$ depends only on $\left\{
\varphi_{y}\right\}  _{y\in\partial_{2}B\cap A},$ where $\partial
_{2}B:=\left\{  y\notin B:\mathrm{d}\left(  y,B\right)  \leq2\right\}  $. As
the measures are Gaussian, for $x\in B$ one has that $E_{A}\left[  \varphi_{x}|\mathcal{F}_{A\setminus
B}\right]  $ is a linear combination of the variables $\left\{  \varphi
_{y}\right\}  _{y\in\partial_{2}B\cap A}$.

From general properties of Gaussian distributions, one easily gets the
following result.

\begin{proposition}[{\citet[Lemma~2.2]{Cip13}}]
\label{prop:MP} Let $A$ be a finite subset of $\mathbb{Z}^{d},$ and $B\subset
A$, and let $\left\{  \varphi_{x}\right\}  _{x\in\mathbb{Z}^{d}}$ be the
membrane model under the measure $P_{A}$. Let further $\left\{  \varphi
_{x}^{\prime}\right\}_{x\in B}  $ be independent of $\left\{  \varphi_{x}\right\}_{x\in B}  $
and distributed according to $P_{B}$, i.e. with $0$-boundary conditions
outside $B.$ Then $\{\varphi_{x}\}_{x\in B}$ has the same distribution under
$P_{A}$ as $\left\{  E_{A}\left[  \varphi_{x}|\mathcal{F}_{A\setminus
B}\right]  +\varphi_{x}^{\prime}\right\}  _{x\in B}$.
\end{proposition}

\begin{corollary}
\label{Cor_variance_monotone}Let $B\subset A$ be finite subsets of
$\mathbb{Z}^{d},$ and $x_{1},\ldots,x_{k}\in B,\ \lambda_{1},\ldots
,\lambda_{k}\in\mathbb{R}$, then%
\[
\operatorname*{var}\nolimits_{P_{B}}\left(  \sum\nolimits_{i=1}^{k}\lambda
_{i}\varphi_{x_{i}}\right)  \leq\operatorname*{var}\nolimits_{P_{A}}\left(
\sum\nolimits_{i=1}^{k}\lambda_{i}\varphi_{x_{i}}\right)  .
\]

\end{corollary}

\begin{proof}
By the previous proposition, $\sum\nolimits_{i=1}^{k}\lambda_{i}\varphi
_{x_{i}}$ has under $P_{A}$ the same law as%
\[
E_{A}\left[  \sum\nolimits_{i=1}^{k}\lambda_{i}\varphi_{x_{i}}|\mathcal{F}%
_{A\setminus B}\right]  +\sum\nolimits_{i=1}^{k}\lambda_{i}\varphi_{x_{i}%
}^{\prime}%
\]
where $\left\{  \varphi_{x}^{\prime}\right\}  _{x\in B}$ is independent of the
first summand and distributed according to $P_{B}.$ From that, the claim follows.
\end{proof}

For $A\subset W\subset\mathbb{Z}^{d}$ we write $P_{W}^{A}:=P_{W\backslash A},$ i.e. the
membrane model with $0$-boundary conditions on both $W^{\c}$ and on $A.$ We use $E^A_W$ to denote the average with respect to $P_W^A$. We also
write $G_{W}^{A}$ for the corresponding covariance matrix. If $A=\emptyset,$
then $P_{W}^{\emptyset}=P_{W}.$ Again, we just use the index $N$ if $W=V_{N}.$

\begin{corollary}
Let $A\subset\mathbb{Z}^{d}$, and $d\geq5.$ Then the weak limit $P^{A}%
:=\lim_{N\rightarrow+\infty}P_{N}^{A}$ exists, and it is a centered Gaussian field,
with covariances%
\[
G^{A}\left(  x,\,y\right)  =\lim_{N\rightarrow+\infty}G_{N}^{A}\left(
x,\,y\right) ,\quad x,\,y\in \Z^d .
\]

\end{corollary}

\begin{proof}
By Corollary \ref{Cor_variance_monotone}, $G_{N}^{A}\left(  x,x\right)
\uparrow G^{A}\left(  x,x\right)  <+\infty$ for all $x,$ as $N\rightarrow
+\infty.$ The finiteness comes from $G_{N}^{A}\left(  x,x\right)  \leq
G_{N}\left(  x,x\right)  \leq G\left(  x,x\right)  <+\infty$ (recall \eqref{eq:GtoGamma}). So $\left\{
P_{N}^{A}\right\}  _{N}$ is a tight sequence. But for $x,\,y\in\mathbb{Z}^{d},$
also $\lim_{N\rightarrow+\infty}\operatorname*{var}_{P_{N}^{A}}\left(
\varphi_{x}+\varphi_{y}\right)  $ exists, and therefore $\lim_{N\rightarrow
+\infty}G_{N}^{A}\left(  x,\,y\right)  $ exists. This implies the statement of
the corollary.
\end{proof}
\paragraph{Bernoulli domination.}A key step of our proof is that the
environment of pinned points can be compared with Bernoulli site percolation.
Expanding $\prod_{x\in W}\left(  \mathrm{d}\varphi_{x}+\varepsilon\delta
_{0}(\mathrm{d}\varphi_{x})\right)  $ in (\ref{def:pinned}), one has, for any
measurable function $f:\mathbb{R}^{\mathbb{Z}^{d}}\rightarrow\mathbb{R}$,
\begin{align*}
E_{W}^{\varepsilon}(f)  &  =\frac{1}{Z_{W}^{\varepsilon}}\int f(\varphi
)\exp\left(  -\frac{1}{2}\left\langle \Delta\varphi,\Delta\varphi\right\rangle
\right)  \prod_{x\in W}\left(  \mathrm{d}\varphi_{x}+\varepsilon\delta
_{0}(\mathrm{d}\varphi_{x})\right)  \prod_{x\in W^{c}}\delta_{0}%
(\mathrm{d}\varphi_{x})=\\
&  =\sum_{A\subset W}\varepsilon^{|A|}\frac{Z_{W}^{A}}{Z_{W}^{\varepsilon}%
}E_{W}^{A}(f),
\end{align*}
where $Z_{W}^{A}:=Z_{W\backslash A}$ i.e.%
\[
P_{W}^{\varepsilon}=\sum_{A\subset W}\zeta_{W}^{\varepsilon}(A)P_{W}^{A}.
\]
with
\begin{equation}
\zeta_{W}^{\varepsilon}(A):=\varepsilon^{|A|}\frac{Z_{W}^{A}}{Z_{W}%
^{\varepsilon}}, \label{eq:def_zeta}%
\end{equation}
which is a probability measure on $\mathcal{P}(W)$, the set of subsets of $W.$
We will often use $\mathcal{A}$ or $\mathcal{A}_{W}$ to denote a $\mathcal{P}(W)$-valued random
variable with this distribution, so that we can write%
\begin{equation}
E_{W}^{\varepsilon}[\varphi_{x}\varphi_{y}]=\sum_{A\subset W}\zeta
_{W}^{\varepsilon}(A)G_{W}^{A}\left(  x,\,y\right)  =E_{\zeta_{W}^{\varepsilon}%
}\left(  G_{W}^{\mathcal{A}}\left(  x,\,y\right)  \right).  \label{eq:mixture1}%
\end{equation}

\begin{lemma}
\label{lem:partition} In $d\geq5$ there exist constants $0<{C}_{-},\,{C}%
_{+}<\infty$ depending only on the dimension such that for every $w\in W$ and
$E\subset W\setminus\{w\}$
\begin{equation}
{C}_{-}\leq\frac{Z_{W}^{E\cup\left\{  w\right\}  }}{Z_{W}^{E}}\leq{C}_{+}.
\label{eq:boundZ}%
\end{equation}

\end{lemma}

\begin{proof} The proof follows the ideas of \citet[Section 5.3]{velenikloc}.
$Z_{W}^{E\cup\left\{  w\right\}  }/Z_{W}^{E}$ is the density at $0$ of the
distribution of $\varphi_{w}$ under the law $P_{W}^{E}$, i.e.%
\[
\frac{Z_{W}^{E\cup\left\{  w\right\}  }}{Z_{W}^{E}}=\frac{1}{\sqrt{2\pi
G_{W}^{E}\left(  w,w\right)  }}.
\]
As%
\[
0<G_{\left\{  w\right\}  }\left(  w,w\right)  \leq G_{W}^{E}\left(
w,w\right)  \leq G\left(  w,w\right)  =G\left(  0,0\right)  <+\infty,
\]
the claim follows.
\end{proof}

\begin{remark}
For $d=2,\,3,\,4,$ one has a similar upper bound for $Z_{W}^{E\cup\left\{
w\right\}  }/Z_{W}^{E}$, but the lower bound depends on $W,$ as $G\left(
0,0\right)  =+\infty.$ For $d=4,$ one has, for $W:=V_N$,%
\[
\frac{Z_{N}^{E\cup\left\{  w\right\}  }}{Z_{N}^{E}}\geq\frac{C_{-}}{\sqrt{\log
N}}.
\]

\end{remark}

We control now the pinning measure $\zeta_{N}^{\varepsilon}$ through
dominations by Bernoulli product measures.

\begin{definition}[Strong stochastic domination]
\label{Def_Domination}Given two probability measures $\mu$ and $\nu$ on the
set $\mathcal{P}(W)$, $\left\vert W\right\vert <+\infty$, we say
that $\mu$ \emph{dominates} $\nu$ \emph{strongly stochastically} if for all $x\in W$,
$E\subset W\setminus\{x\}$,
\begin{equation}
\mu(A:\,x\in A\,|\,A\setminus\{x\}=E)\geq\nu(A:\,x\in A\,|\,A\setminus
\{x\}=E).\label{eq:ssd}%
\end{equation}
When this holds we write $\mu\succ\nu$.
\end{definition}

Let $\mathbb{P}_{W}^{\rho}$ be the Bernoulli site percolation measure on $W$
with intensity $\rho\in\left[  0,1\right]  .$ We regard this as a probability
measure on $\mathcal{P}\left(  W\right)  .$

\begin{proposition}
\label{prop:dom} Let $d\geq5$ and $\varepsilon>0$. Then%
\[
\mathbb{P}_{W}^{\rho_{-}(d,\varepsilon)}\prec\zeta_{W}^{\varepsilon}%
\prec\mathbb{P}_{W}^{\rho_{+}(d,\varepsilon)}%
\]
with%
\begin{equation}
\rho_{\pm}\left(  d,\varepsilon\right)  :=\frac{C_{\pm}\left(  d\right)
\varepsilon}{1+C_{\pm}\left(  d\right)  \varepsilon}\in\left(  0,1\right)
\label{Def_rho}%
\end{equation}
where $C_{-},\,C_{+}$ are defined in Lemma~\ref{lem:partition}.
\end{proposition}

\begin{proof}
For $x,\,E$ as in Definition \ref{Def_Domination}%
\[
\zeta_{W}^{\varepsilon}(A:\,x\in A\,|\,A\setminus\{x\}=E)=\left[
1+\frac{1}{\varepsilon}\frac{Z_{W}^{E}}{Z_{W}^{  E\cup\left\{  x\right\}
  }}\right]  ^{-1}.
\]
This proves the claim.
\end{proof}

\section{Proof of the main result\label{sec:main}}

\subsection{Sobolev norms\label{Subsect_Sobolev}}

A crucial role of the proof uses a Sobolev-type norm $\left\Vert
\cdot\right\Vert _{A,E}$ depending on subsets $A\subset E\subset\mathbb{Z}^{d}$.
Given $A$, let
\[
\widehat{A}:=\left\{  x\in A:x+e\in A,\,\text{for all } \,e\right\}  .
\]
$\widehat{A}$ is the subset of \textquotedblleft interior\textquotedblright%
\ points of $A$. For $f:\mathbb{Z}^{d}\rightarrow\mathbb{R}$ and
$A\subset E\subset\mathbb{Z}^{d}$, let
\begin{equation}
\Vert f\Vert_{A,\,E}^{2}:=\sum_{x\in E}\frac{f(x)^{2}}{1+\mathrm{d}%
(x,\,\widehat{A})^{2d+3}}+\sum_{x\in E}\frac{\left\Vert \nabla f\left(
x\right)  \right\Vert ^{2}}{1+\mathrm{d}(x,\,\widehat{A})^{d+2}}+\sum_{x\in
E}\left\Vert \nabla^{2}f\left(  x\right)  \right\Vert ^{2}.
\label{eq:def_norm}%
\end{equation}
If $\widehat{A}=\emptyset,$ then we put $\mathrm{d}(x,\,\widehat{A})=+\infty$
by convention, and $\Vert f\Vert_{A,\,E}^{2}=\sum_{x\in E}\left\Vert
\nabla^{2}f\left(  x\right)  \right\Vert ^{2}.$ We note the following two facts:

\begin{enumerate}
\item $\|f\|^{2}_{A,\,E}$ is defined for $f:\,E\cup\partial_{2} E\to\mathbb{R}
$.

\item If $E_{1}$ and $E_{2}$ are disjoint then
\[
\Vert f\Vert_{A,\,E_{1}\cup E_{2}}^{2}=\Vert f\Vert_{A,\,E_{1}}^{2}+\Vert
f\Vert_{A,\,E_{2}}^{2}.
\]

\end{enumerate}

{When $E:=\Z^d$ and we randomize $A$ thinking of it as the set of pinned points, we will use this norm as the random Sobolev norm equivalent to $\|\nabla^2\cdot\|_{L^2}$}. We now bound the $\Vert\cdot\Vert_{A,\,\mathbb{Z}^{d}}^{2}$ norm of a function
vanishing on $A$ by second derivates only.

\begin{lemma}
\label{lem:equivalence} Let $f$ be a function which is identically zero on
$A$. Then%
\[
\Vert f\Vert_{A,\,\Z^d}^{2}\leq L\sum_{x\in \Z^d}\left\Vert \nabla^{2}f\left(
x\right)  \right\Vert ^{2}.
\]

\end{lemma}

\begin{proof}
There is nothing to prove when $\widehat{A}=\emptyset,$ so we assume
$\widehat{A}\neq\emptyset.$

We first show that the first summand on the right hand side of
(\ref{eq:def_norm}) is dominated by a multiple of the second, and afterwards
that the second is dominated by the third.

If $x\in\mathbb{Z}^{d}$, we choose a nearest-neighbor path $\psi_{x}$ of
shortest length $|\psi_x|:=k+1$ to $\widehat{A}$, that is, $\psi_{x}=\left(  x_{0}%
=x,x_{1},\ldots,x_{k}\right)  $ with $x_{k}\in\widehat{A}$. As $f$ is $0$ on
$A$, one has
\[
f(x)=\sum_{l=1}^{k}(f(x_{l-1})-f(x_{l})).
\]
We can choose the collection $\left\{  \psi_{x}\right\}  $ of paths in such a
way that the same bond is not used for two different end points in
$\widehat{A}$. More formally: if $x,x^{\prime}\in\mathbb{Z}^{d}$ with paths
$\psi_{x}=\left(  x,x_{1},\ldots,x_{k}\right)  ,\ \psi_{x^{\prime}}=\left(
x^{\prime},x_{1}^{\prime},\ldots,x_{k^{\prime}}^{\prime}\right)  $ have the
property that there exists a bond $b$ which belongs to both paths, then
$x_{k}=x_{k^{\prime}}^{\prime}.$ This can be achieved by choosing an
enumeration $\left\{  x_{n}\right\}  $ of $\mathbb{Z}^{d},$ and constructing
the paths recursively with this property.

By Cauchy-Schwarz,
\[
f(x)^{2}\leq\left\vert \psi_{x}\right\vert \sum_{l=1}^{k}(f(x_{l}%
)-f(x_{l-1}))^{2}=\mathrm{d}(x,\widehat{A})\sum_{l=1}^{\mathrm{d}%
(x,\widehat{A})}(f(x_{l})-f(x_{l-1}))^{2},
\]
and thus, exchanging the order of summation between points $x$ and paths
$\psi_{x}$,
\begin{align}
\sum_{x}\frac{f\left(  x\right)  ^{2}}{1+\mathrm{d}(x,\widehat{A})^{2d+3}}  &
\leq\sum_{x}\frac{\mathrm{d}(x,\widehat{A})}{1+\mathrm{d}(x,\widehat
{A})^{2d+3}}\sum_{l=1}^{\mathrm{d}(x,\widehat{A})}(f(x_{l})-f(x_{l-1}%
))^{2}\nonumber\\
&  \leq\sum_{z}\left\Vert \nabla f\left(  z\right)  \right\Vert ^{2}%
\sum_{x:z\in\psi_{x}}\frac{\mathrm{d}(x,\widehat{A})}{1+\mathrm{d}%
(x,\widehat{A})^{2d+3}}. \label{eq:use_two}%
\end{align}
For $z\in\mathbb{Z}^{d}$ write $R_{z,k}:=\{x\in\mathbb{Z}^{d}:\mathrm{d}%
(x,\widehat{A})=k\;\text{and}\;z\in\psi_{x}\}.$ {Observe that every $x\in\mathbb{Z}^{d}$ with $z\in\psi_{x}$ satisfies
$\mathrm{d}(x,\widehat{A})\geq\mathrm{d}(z,\widehat{A})$. Notice also that if $z\in \psi_x$, the path $\psi_x$ cannot take less than $\De (z,\,\widehat A)$ steps to reach $\widehat A$ from $z$ (otherwise $\De(z,\,\widehat A)$ would not be minimal). Thus $R_{z,\,k}$ can be bounded by the volume of a ball around $z$, namely,} there exists a
constant $c_{1}=c_{1}(d)$ such that $\left\vert R_{z,k}\right\vert \leq
c_{1}(k-\mathrm{d}(z,\widehat{A}))^{d-1}\leq c_{1}k^{d-1}.$ Therefore
\begin{align}
\sum_{x:z\in\psi_{x}}\frac{\mathrm{d}(x,\widehat{A})}{1+\mathrm{d}%
(x,\widehat{A})^{2d+3}}  &  \leq\sum_{k=\mathrm{d}(z,\widehat{A})}^{\infty
}\frac{|R_{z,k}|}{1+k^{2d+2}}\nonumber\\
&  \leq L\sum_{k=\mathrm{d}(z,\widehat{A})}^{\infty}\frac{1}{1+k^{d+3}}\leq
L\frac{1}{1+\mathrm{d}(z,\widehat{A})^{d+2}}. \label{eq:use_one}%
\end{align}
Thus we have, plugging \eqref{eq:use_one} in \eqref{eq:use_two},%
\[
\sum_{x}\frac{f(x)^{2}}{1+\mathrm{d}(x,\,\widehat{A})^{2d+3}}\leq L\sum
_{x}\frac{1}{1+\mathrm{d}(x,\,\widehat{A})^{d+2}}\left\Vert \nabla f\left(
x\right)  \right\Vert ^{2}.
\]
It remains to prove that the right hand side is bounded by some multiple
of $\sum_{x}\left\Vert \nabla^{2}f\left(  x\right)  \right\Vert ^{2}.$ If
$\psi_{x}$ is the same as above, we have%
\[
\nabla f(x)=\sum_{l=1}^{k}(\nabla f(x_{l-1})-\nabla f(x_{l})),
\]
because $\nabla f(x_{k})=0$ component-wise for $x_{k}\in\widehat{A}$. By the same
arguments as above we get
\[
\left\Vert \nabla f\left(  x\right)  \right\Vert ^{2}\leq\mathrm{d}%
(x,\widehat{A})\sum_{l=1}^{\mathrm{d}(x,\widehat{A})}\left\vert \nabla\left[
f(x_{l})-f(x_{l-1})\right]  \right\vert ^{2},
\]
and
\begin{align*}
&  \sum_{x}\frac{\left\Vert \nabla
f\left(  x\right)  \right\Vert ^{2}}{1+\mathrm{d}(x,\widehat{A})^{d+2}}\leq\sum_{x\in\mathbb{Z}^{d}}%
\frac{\mathrm{d}(x,\widehat{A})}{1+\mathrm{d}(x,\widehat{A})^{d+2}}\sum
_{l=1}^{k}\left\Vert \nabla\left[  f\left(  x_{l}\right)  -f\left(
x_{l-1}\right)  \right]  \right\Vert ^{2}\\
&  \leq L\sum_{z}\left\Vert \nabla^{2}f\left(  z\right)  \right\Vert ^{2}%
\sum_{x:z\in\psi_{x}}\frac{1}{1+\mathrm{d}(x,\widehat{A})^{d+1}}\leq L\sum
_{z}\left\Vert \nabla^{2}f\left(  z\right)  \right\Vert ^{2}\left[  \sup
_{y\in\widehat{A}}\sum_{x}\frac{1}{1+\mathrm{d}(x,\,y)^{d+1}}\right] \\
&  \leq L\sum_{z}\left\Vert \nabla^{2}f\left(  z\right)  \right\Vert ^{2}.
\end{align*}

\end{proof}

For $k\geq0$ and $E\subset\mathbb{Z}^{d}$ let
\[
\upsilon_{k}\left(  E\right)  :=\left\{  x:\mathrm{d}\left(  x,E\right)  \leq
k\right\}  .
\]
For $x,\,y\in\mathbb{Z}^{d}$ let $\Gamma_{x,\,y}$ be the set of non-intersecting
nearest-neighbor paths
\[
\psi=\left(  x_{0}=x,x_{1},\ldots,x_{n}=y\right)  ,
\]
and we write $\ell\left(  \psi\right)  $ for the length $n.$ For such a $\psi$
we define
\begin{equation}
\phi_{A}\left(  \psi\right)  :=\sum_{i=0}^{n}q_{A}\left(  x_{i}\right)  ,
\label{Def_path_weight}%
\end{equation}
where
\[
q_{A}\left(  x\right)  :=\frac{1}{1+\mathrm{d}\left(  x,\widehat{A}\right)
^{2d+3}},\quad x\in\mathbb{Z}^{d}.
\]
Define
\begin{align*}
\widehat{\mathrm{d}}_{A}\left(  x,\,y\right)   &  :=\min\left\{  \phi_{A}\left(
\psi\right)  :\psi\in\Gamma_{x,\,y}\right\}  ,\\
\widehat{\mathrm{d}}_{A}(0,\,0)  &  :=0.
\end{align*}
{$\widehat d_A$ is defined in such a way that the shortest weight $\phi_A$ is achieved by staying far off pinned points.}. See Figure~\ref{fig:hatDistance} for a $2$-dimensional example of $\phi_A$ for paths between a point $y$ and the origin.

$\widehat{\mathrm{d}}_{A}$ may well be bounded, for instance if $A$ is a
finite set. In the cases we are interested in, it will however be unbounded. We will often just write $\widehat{\mathrm{d}}$ if it is clear from the
context what set $A$ is considered. Since $q_{A}(x)\leq1$ for any $x$, note also the bound $\widehat{\mathrm{d}%
}(x,\,y)\leq\mathrm{d}(x,\,y)$ for all $x,\,y\in\mathbb{Z}^{d}$.
\begin{figure}[ht!]
\includegraphics[scale=1.2]{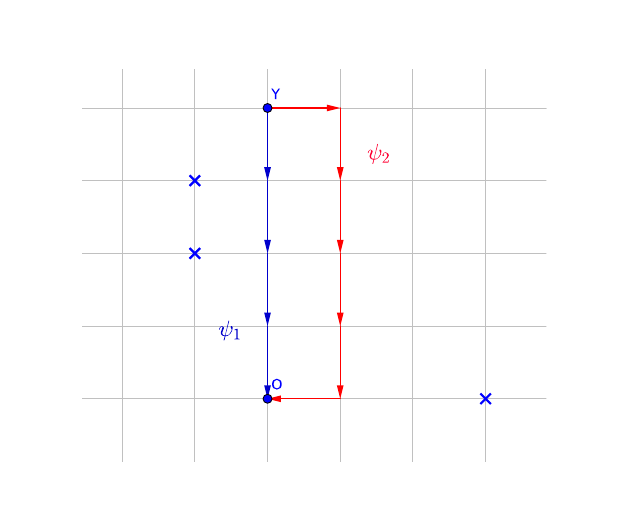}
\caption{Crosses represent pinned points. Note that $\psi_1$ is smaller in graph distance than  $\psi_2$. However, since $\psi_1$ is closer to pinned points, $1.015\approx \phi_A(\psi_1)>\phi_A(\psi_2)\approx 0.032$. The $\widehat\De_A$-distance between $y$ and $O$ is then achieved minimizing $\phi_A$ over all paths.}
\label{fig:hatDistance}
\end{figure}

We define%
\begin{equation}
C_{n}:=\left\{  x:\,\widehat{\mathrm{d}}\left(  0,x\right)  \leq10n\right\}  .
\label{eq:C_n}%
\end{equation}
$C_{n}$ is connected in the usual graph structure of $\mathbb{Z}^{d},$ but the
complement may be disconnected. If we want to emphasize the dependence of $C_{n}$ on $A$, we write $C_{n,A}.$ ]textcolor{red}{Note that the fewer the pinned points $\widehat A$ in a region, the larger the sets $C_n$ are.}

\begin{remark}\label{rem:intersection_empty}
Remark that $\upsilon_{2}(C_{n})\cap\upsilon_{2}(C_{n+1}^{\c})=\emptyset$. In
fact, assuming that there is a $w\in\upsilon_{2}(C_{n})\cap\upsilon
_{2}(C_{n+1}^{\c})$, then there exist 
\begin{equation}
\label{eq:contradict}
w_{1}\in C_{n},\quad w_{2}\in
C_{n+1}^{\c}
\end{equation}
with $\mathrm{d}(w,\,w_{i})\leq2$ for $i=1,\,2$. Hence
$\widehat{\mathrm{d}}(w_{1},\,w_{2})\leq\mathrm{d}(w_{1},\,w_{2})\leq 4$
by the triangle inequality for the graph distance, which contradicts \eqref{eq:contradict}. 
\end{remark}

We will need a monotonicity property in the dependence on $A$. First remark
that if $A\subset A^{\prime}$ then $\mathrm{d}\left(  x,\widehat{A}\right)
\geq\mathrm{d}\left(  x,\widehat{A^{\prime}}\right)  $ for all $x$, and
therefore%
\begin{equation}
\widehat{\mathrm{d}}_{A}
\leq\widehat{\mathrm{d}}_{A^{\prime}}. \label{monoton1}%
\end{equation}

\begin{lemma}
\label{lem:interpolation}For every $n$, there exists a function $\eta
_{n}:\mathbb{Z}^{d}\rightarrow\left[  0,1\right]  $ with the following properties:
\begin{equation}
\eta_{n}=0\ \mathrm{on\ }\upsilon_{2}\left(  C_{n}\right)  ,\ \eta
_{n}=1\ \mathrm{on\ }\upsilon_{2}\left(  C_{n+1}^{\c}\right),  \label{prop_eta1}%
\end{equation}%
\begin{equation}
\left\Vert \nabla\eta_{n}\left(  x\right)  \right\Vert _{\infty}\leq \frac
{L}{1+\mathrm{d}(x,\widehat{A})^{2d+3}},\ \forall \,x\in\mathbb{Z}^{d}.
\label{prop_eta2}%
\end{equation}

\end{lemma}

\begin{proof}
Let $f_{1}(x):=\widehat{\mathrm{d}}(x,\upsilon_{2}(C_{n}))$ and $f_{2}%
(x):=\widehat{\mathrm{d}}(x,\upsilon_{2}(C_{n+1}^{\c})).$ We define
\[
\eta_{n}(x):=\frac{f_{1}(x)}{f_{1}(x)+f_{2}(x)}.
\]
which evidently satisfies (\ref{prop_eta1}).

To prove (\ref{prop_eta2}), notice first that one can find an $L$ large uniformly for all $x$ with $\mathrm{d}(x,\widehat{A})\le 4,$ so let us consider
$x$ such that $\mathrm{d}(x,\widehat{A})\geq5.$ We have from Remark~\ref{rem:intersection_empty} that
\[
f_{1}(x)+f_{2}(x)\geq1.
\]%
Then
\begin{align}
\left\vert D_{e}\eta_{n}\left(  x\right)  \right\vert  &  \leq\frac{\left\vert
D_{e}f_{1}\left(  x\right)  \right\vert }{f_{1}\left(  x\right)  +f_{2}\left(
x\right)  }\nonumber\\
&  +{f_{1}\left(  x+e\right)  }\left\vert \frac{1}{f_{1}\left(  x\right)
+f_{2}\left(  x\right)  }-\frac{1}{f_{1}\left(  x+e\right)  +f_{2}\left(
x+e\right)  }\right\vert . \label{interpo1}%
\end{align}
We see that
\begin{align*}
\left\vert D_{e}f_{1}\left(  x\right)  \right\vert  &  \leq q_{A}%
(x)+q_{A}(x+e)\\
&  \leq\frac{2}{1+\min\left\{  \mathrm{d}\left(  x,\,\widehat{A}\right)
,\,\mathrm{d}\left(  x+e,\,\widehat{A}\right)  \right\}  ^{2d+3}}\\
&  \leq\frac{L}{1+\mathrm{d}\left(  x,\,\widehat{A}\right)  ^{2d+3}},
\end{align*}
as we assumed $\mathrm{d}(x,\widehat{A})\geq5$. The same estimate is true also
for $\left\vert D_{e}f_{2}\left(  x\right)  \right\vert .$

The second summand in (\ref{interpo1}) is bounded above by $\left\vert
D_{e}f_{1}\left(  x\right)  \right\vert +\left\vert D_{e}f_{2}\left(
x\right)  \right\vert ,$ so the claim follows.
\end{proof}

\begin{corollary}
\label{Rem_eta}
For all $
x\in\mathbb{Z}^{d}$ it holds that
\begin{enumerate}
\item[a)] for all $k\ge 2$ there exists $L=L(k)>0$ such that
\begin{equation}
\left\Vert \nabla^{k}\eta_{n}\left(  x\right)  \right\Vert _{\infty}\leq
\frac{L }{1+\mathrm{d}(x,\widehat{A})^{2d+3}}. \label{prop_etak}%
\end{equation}
\item[b)] For all $e$ neighbors of the origin and $k\ge 1$ there exists $L=L(k)>0$ such that
\begin{equation}
\left\Vert \nabla^k\eta_{n}\left(  x+e\right)  \right\Vert _{\infty}\leq\frac
{L}{1+\mathrm{d}(x,\widehat{A})^{2d+3}},\ \forall x\in\mathbb{Z}^{d},\ \forall
e. \label{prop_eta_shift}%
\end{equation}
\end{enumerate}
\end{corollary}

\begin{proof}\noindent
\begin{enumerate}
\item[a)] (\ref{prop_eta2}) implies that also higher order derivatives can be
estimated by the same bound with a changed $L$ because the supremum norm of higher
order discrete derivatives can be estimated by the first order ones.%
\item[b)] Again this holds by an estimate with first order derivatives and the fact that 
$$\left\vert \mathrm{d}(x+e,\widehat{A})-\mathrm{d}(x,\widehat
{A})\right\vert \leq1.$$
\end{enumerate}
\end{proof}
Consider now an infinite set $A$ with the property that $C_{n,A}$ is finite
for all $n.$
Given $A$ with $A^{\c}$ finite, and $0\notin A$, we consider the unique
function $h_{A}$ which satisfies $h_{A}\left(  x\right)  =0$ on $A,$ and for all $x\in A^{\c}$
\[
\Delta^{2}h_{A}\left(  x\right)  =\delta_{0}(x).
\]

\begin{lemma}
\label{lem_bound_Sobolev}With the above notation, we have for $n\geq1$%
\[
\Vert h_{A}\Vert_{A,C_{n+1}^{\c}}^{2}\leq L\Vert h_{A}\Vert_{A,C_{n+1}\setminus
C_{n}}^{2}.
\]
It is important to emphasize that $L$ depends neither on $A$ nor on $n$.
\end{lemma}

\begin{proof}
Fix $n,$ and let $\eta_{n}$ be as in Lemma \ref{lem:interpolation}. We also drop the subscript $A$ in $h_A$. We have
with Lemma \ref{lem:equivalence} and Lemma \ref{lem:sumbyparts}%
\begin{align}
\Vert h\Vert_{A,C_{n+1}^{\c}}^{2}  &  =\Vert\eta_{n}h\Vert_{A,C_{n+1}^{\c}}%
^{2}\leq\Vert\eta_{n}h\Vert_{A,\mathbb{Z}^{d}}^{2}\leq L\sum_{e,e^{\prime}%
}\left\langle D_{e}D_{e^{\prime}}\eta_{n}h,D_{e}D_{e^{\prime}}\eta
_{n}h\right\rangle \nonumber\\
&  =L\left\langle \eta_{n}h,\Delta^{2}\left(  \eta_{n}h\right)  \right\rangle.\label{eq:scalar_prod}
\end{align}
By an elementary computation, one has for any $f,g:\mathbb{Z}^{d}%
\rightarrow\mathbb{R}$ and $x\in\mathbb{Z}^{d}$%
\begin{equation}
\Delta\left(  fg\right)  \left(  x\right)  =f\left(  x\right)  \Delta g\left(
x\right)  +\Delta f\left(  x\right)  g\left(  x\right)  +\frac{1}{2d}\sum
_{e}D_{e}f\left(  x\right)  D_{e}g\left(  x\right)  . \label{product_rule}%
\end{equation}
Applying this twice gives%
\begin{align*}
\Delta^{2}(\eta_{n}h)  &  =\eta_{n}\Delta^{2}h+\left(  \Delta^{2}\eta
_{n}\right)  h+2\Delta\eta_{n}\Delta h+\frac{1}{d}\sum_{e}\left(D_{e}\Delta\eta
_{n} \right)D_{e}h\\
&  +\frac{1}{d}\sum_{e}\left(D_{e}\eta_{n}\right) \left(D_{e}\Delta h\right)+\frac{1}{4d^{2}}%
\sum_{e,e^{\prime}}\left(D_{e^{\prime}}
D_{e}\eta_{n}\right) D_{e^{\prime}}D_{e}h\\
&  =:F_{1}+F_{2}+2F_{3}+\frac{1}{d}F_{4}+\frac{1}{d}F_{5}+\frac{1}{4d^{2}%
}F_{6}.%
\end{align*}%
Note that
\begin{equation}
\left\langle \eta_{n}h,F_{1}\right\rangle =\left\langle \eta_{n}h,\eta
_{n}\Delta^{2}h\right\rangle =0, \label{F1}%
\end{equation}
as for $x\neq0$ we have $\Delta^{2}h\left(  x\right)  =0$ and for $x=0$ we have $\eta_{n}\left(  0\right)  =0$.
All the other terms contain derivatives of $\eta_{n}$. Therefore, every derivative of the function $\eta_n$ will be non-zero only for points in $C_{n+1}%
\setminus C_{n}.$
Since we have \eqref{eq:scalar_prod}, we need to estimate $\langle \eta_n h,\,F_i \rangle$ for $i=2,\,\ldots,\,6$. Let us begin with $i=2$: by Corollary \ref{Rem_eta}
\begin{align}
\left\vert \left\langle \eta_{n}h,F_{2}\right\rangle \right\vert  &  \leq
\sum_{x}\left\vert \Delta^{2}\eta_{n}\left(  x\right)  \right\vert h\left(
x\right)  ^{2}  =\sum_{x\in C_{n+1}\backslash C_{n}}\left\vert \Delta^{2}\eta_{n}\left(
x\right)  \right\vert h\left(  x\right)  ^{2}\nonumber\\
&  \leq\sum_{x\in C_{n+1}\backslash C_{n}}\frac{L}{1+\mathrm{d}(x,\widehat
{A})^{2d+3}}h\left(  x\right)  ^{2}  \leq L\Vert h\Vert_{A,C_{n+1}\setminus C_{n}}^{2}.\label{F2}
\end{align}
Let us see now $i=3$. With the Cauchy-Schwarz inequality we get
\begin{align}
\left\vert \left\langle \eta_{n}h,F_{3}\right\rangle \right\vert  &
\leq L\left\vert \sum\nolimits_{x\in C_{n+1}\backslash C_{n}}\eta_{n}\left(
x\right)  h\left(  x\right)  \Delta\eta_{n}\left(  x\right)  \Delta h\left(
x\right)  \right\vert \nonumber\\
&  \leq L\sqrt{\sum\nolimits_{x\in C_{n+1}\backslash C_{n}}\left(  \Delta
h\left(  x\right)  \right)  ^{2}}\sqrt{\sum\nolimits_{x\in C_{n+1}\backslash
C_{n}}\left(  \Delta\eta_{n}\left(  x\right)  \right)  ^{2}h\left(
x\right)  ^{2}}\nonumber\\
&  \leq L\Vert h\Vert_{A,C_{n+1}\setminus C_{n}}^{2}\label{F3}
\end{align}
using Corollary~\ref{Rem_eta}, \eqref{eq:char_lapl} and the arithmetic-geometric mean inequality.

To estimate the part with $F_{4},$ we first observe that $D_{e}\Delta\eta
_{n}\left(  x\right)  $ is $0$ outside $C_{n+1}\backslash C_{n},$ and by
Remark \ref{Rem_eta}%
\[
\left\vert D_{e}\Delta\eta_{n}\left(  x\right)  \right\vert \leq\frac
{L}{1+\mathrm{d}(x,\widehat{A})^{2d+3}}.
\]
Therefore, using the inequality of arithmetic and geometric means,%
\begin{align}
\left\vert \left\langle \eta_{n}h,F_{4}\right\rangle \right\vert  &  \leq
L\sum_{e} \sqrt{\sum_{x\in C_{n+1}\backslash C_{n}}\frac
{h(x)^2}{\left(1+\mathrm{d}(x,\widehat{A})^{2d+3}\right)^2}}  \sqrt{\sum_{x\in C_{n+1}\backslash C_{n}}\frac
{D_{e}h\left(  x\right)^2 }{\left(1+\mathrm{d}(x,\widehat{A})^{2d+3}\right)^2}} 
\nonumber\\
&  \leq L\sum_{e}\sqrt{\sum_{x\in C_{n+1}\backslash C_{n}}\frac
{h(x)^2}{1+\mathrm{d}(x,\widehat{A})^{2d+3}}}\sqrt{\sum_{x\in C_{n+1}\backslash C_{n}}\frac
{D_{e}h\left(  x\right)^2 }{1+\mathrm{d}(x,\widehat{A})^{d+2}}}\nonumber\\
&  \leq L\Vert h\Vert_{A,C_{n+1}\setminus C_{n}}^{2}.\label{F4}
\end{align}
For the estimate of $\left\langle \eta_{n}h,\,F_5\right\rangle $ we can use Lemma~\ref{lem:interpolation} and \eqref{eq:char_lapl} again to say that, for a fixed direction $e$,
\begin{align*}
&\left\vert \left\langle \eta_{n}h,\,\left(D_{e}\eta_{n}\right) \left(D_{e}\Delta
h\right)\right\rangle \right\vert    \\
&  \leq\left\vert \sum_{x\in C_{n+1}\backslash C_{n}}\eta_{n}\left(  x\right)
h\left(  x\right)  D_{e}\eta_{n}\left(  x\right)  \Delta h\left(  x\right)
\right\vert   +\left\vert \sum_{x\in C_{n+1}\backslash C_{n}}\eta_{n}\left(  x\right)
h\left(  x\right)  D_{e}\eta_{n}\left(  x\right)  \Delta h\left(  x+e\right)
\right\vert \\
&  \leq\sqrt{\sum_{x\in C_{n+1}\backslash C_{n}}h\left(  x\right)  ^{2}%
D_{e}\eta_{n}\left(  x\right)  ^{2}}\left[  \sqrt{\sum_{x\in C_{n+1}\backslash C_{n}}\Delta h\left(
x\right)  ^{2}}+\sqrt{\sum_{x\in C_{n+1}\backslash C_{n}}\Delta h\left(
x+e\right)  ^{2}}\right] \\
&  \leq L\Vert h\Vert_{A,C_{n+1}\setminus C_{n}}^{2}.
\end{align*}
Summing over $e$ yields%
\begin{equation}
\left\vert \left\langle \eta_{n}h,F_{5}\right\rangle \right\vert \leq L\Vert
h\Vert_{A,C_{n+1}\setminus C_{n}}^{2}. \label{F5}%
\end{equation}
It finally remains to show
\begin{equation}
\left\vert \left\langle \eta_{n}h,F_{6}\right\rangle \right\vert \leq L\Vert
h\Vert_{A,C_{n+1}\setminus C_{n}}^{2} \label{F6}%
\end{equation}
which follows in the same way as (\ref{F3}).

Combining (\ref{F1})-(\ref{F6}) proves the lemma.
\end{proof}

With these preparations, we can now prove that $\|h\|^{2}_{A, C_{n+1}^{\c}}$
decays exponentially.

\begin{lemma}
\label{lem:deterministic} Let $d\ge 1$, and let $A\subset
\mathbb{Z}^{d}\setminus\{0\}$ be such that $A^{\c}$ is finite. There exist
constants $c_{1}\left(  d\right)  >0$ and $\delta\left(  d\right)  >0$,
independent of $A$, such that, for all $n\in\mathbb{N},$
\[
\Vert h\Vert_{A,C_{n+1}^{\c}}^{2}\leq c_{1}\mathrm{e}^{-\delta n}\Vert
h\Vert_{A,\mathbb{Z}^{d}}^{2}.%
\]

\end{lemma}

\begin{proof}
From Lemma \ref{lem_bound_Sobolev} we get
\[
\Vert h\Vert_{A,C_{n+1}^{\c}}^{2}\leq L\Vert h\Vert_{A,C_{n+1}\setminus C_{n}%
}^{2}=L\left(  \Vert h\Vert_{A,C_{n}^{\c}}^{2}-\Vert h\Vert_{A,C_{n+1}^{\c}}%
^{2}\right)  ,
\]
that is, iterating the argument,
\begin{align*}
\Vert h\Vert_{A,C_{n+1}^{\c}}^{2}  &  \leq\frac{L}{1+L}\Vert h\Vert
_{A,C_{n}^{\c}}^{2}\leq\left(  \frac{L}{1+L}\right)  ^{n-1}\Vert h\Vert
_{A,C_{1}^{\c}}^{2}\\
&  \leq\frac{1+L}{L}\left(  \frac{L}{1+L}\right)  ^{n}\Vert h\Vert
_{A,\mathbb{Z}^{d}}^{2},
\end{align*}
proving the claim.
\end{proof}

\begin{corollary}
\label{Cor_main}If $d\geq5,$ then, under the same conditions and notation as above%
\[
\Vert h\Vert_{A,C_{n}^{\c}}^{2}\leq c_{1}\mathrm{e}^{-\delta n}.
\]

\end{corollary}

\begin{proof}
By Lemma \ref{lem:equivalence}%
\begin{align*}
\Vert h\Vert_{A,\mathbb{Z}^{d}}^{2}  &  \leq L\sum_{x\in\mathbb{Z}^{d}%
}\left\Vert \nabla^{2}h\left(  x\right)  \right\Vert^2 =L\left\langle
h,\Delta^{2}h\right\rangle   =L h\left(  0\right)  \leq L G(0,\,0)<+\infty.
\end{align*}
Plugging this in Lemma~\ref{lem:deterministic} concludes the proof.
\end{proof}

\subsection{Trapping configurations under the Bernoulli law}

In order to prove our main theorem, we have to obtain probabilistic properties
of the sequence $\left\{  C_{n,\mathcal{A}}\right\}  _{n}$ where $\mathcal{A}$
is random and distributed according to $\zeta^{\varepsilon}.$ Using the
Bernoulli domination, the key probabilistic estimates have to be done only for
a Bernoulli measure instead of $\zeta^{\varepsilon}$. Therefore, let
$p\in\left(  0,1\right)  $ and $\mathbb{P}^{p}$ be the Bernoulli site
percolation measure on the set of subsets of $\mathbb{Z}^{d}$ with parameter
$p.$ As $p$ is fixed in this section, we leave it out in the notation. We
write $\widehat{\mathcal{A}}$ for the set of interior points. Let
$B_{m}\left(  x\right)  :=\left\{  y:\,\De(x,\,y)\leq
m\right\}  $.

\begin{lemma}
\label{Le_box_empty}For $m\in\mathbb{N}$, $x\in\mathbb{Z}^{d},$%
\[
\mathbb{P}\left(  B_{m}\left(  x\right)  \cap\widehat{\mathcal{A}}%
=\emptyset\right)  \leq\left(  1-p^{2d+1}\right)  ^{\left\lfloor \frac{2m+1}{
3}\right\rfloor ^{d} }.
\]

\end{lemma}

\begin{proof}
It suffices to take $x=0$ and write $B_{m}$ for $B_{m}\left(  0\right)  $.
Note that $B_{m}$ is a hypercube of side length $2m+1.$ Put $n:=\left\lfloor
\left(  2m+1\right)  /3\right\rfloor .$ We can place $n^{d}$ pairwise disjoint
boxes $B_{1}\left(  x_{j}\right)  ,\ 1\leq j\leq n^{d}$ in $B_{m}.$ As these
boxes are disjoint, the events $\left\{  x_{j}\in\widehat{\mathcal{A}%
}\right\}  $ are independent and they have probability $p^{2d+1}$. Therefore%
\[
\mathbb{P}\left(  B_{m}\cap\widehat{\mathcal{A}}=\emptyset\right)
\leq\mathbb{P}\left(  x_{j}\notin\widehat{\mathcal{A}},\ \forall\,j\leq
n^{d}\right)  =\left(  1-p^{2d+1}\right)  ^{\left\lfloor \frac{2m+1}%
{3}\right\rfloor ^{d}}.
\]

\end{proof}

\begin{lemma}
\label{Le_main}There exist $\lambda,\,K>0$ and $n_{0}\in \N$ depending only on the
dimension $d$ and $p$ such that for all $n\geq n_{0}$ and all $N\geq Kn$%
\begin{equation}\label{eq:equality}
\mathbb{P}\left(  \sup_{x\in C_{n,\mathcal{A}}}\mathrm{d}\left(  0,x\right)
>Kn\right)  =\mathbb{P}\left(  \inf\limits_{x:\mathrm{d}\left(  x,0\right)
>Kn}\widehat{\mathrm{d}}\left(  0,x\right)  \leq{10}n\right)  \leq
\mathrm{e}^{-\lambda n}%
\end{equation}

\end{lemma}

\begin{proof}
The equality in \eqref {eq:equality} holds by the definition of $C_{n}$. Let us prove the inequality on the right-hand side of the above formula.

For $M\in\mathbb{N}$ we subdivide $\mathbb{Z}^{d}$ in boxes $B_{\mathbf{i}%
},\ \mathbf{i}\in\mathbb{Z}^{d},$ of side-length $M:$%
\[
B_{\mathbf{i}}:=\left(  \left[  \left(  i_{1}-1\right)  M+1,\,i_{1}M\right]
\times\cdots\times\left[  \left(  i_{d}-1\right)  M+1,\,i_{d}M\right]
\right)  \cap\mathbb{Z}^{d},
\]
and%
\[
B_{\mathbf{i}}^{0}:=\left(  \left[  \left(  i_{1}-1\right)  M+2,\,i_{1}%
M-1\right]  \times\cdots\times\left[  \left(  i_{d}-1\right)  M+2,\,i_{d}%
M-1\right]  \right)  \cap\mathbb{Z}^{d},
\]
which is a box contained in $B_{\mathbf{i}}.$ We define%
\[
\eta\left(  \mathbf{i}\right)  =\left\{
\begin{array}
[c]{cc}%
1 & \mathrm{if\ }B_{\mathbf{i}}^{0}\cap\widehat{\mathcal{A}}\neq\emptyset\\
0 & \mathrm{if\ }B_{\mathbf{i}}^{0}\cap\widehat{\mathcal{A}}=\emptyset
\end{array}
\right.  .
\]
The $\eta\left(  \mathbf{i}\right)  $ are i.i.d. In order to estimate
$\mathbb{P}\left(  \eta\left(  \mathbf{i}\right)  =0\right)  $, we subdivide
the box $B_{\mathbf{i}}^{0}$ into boxes $Q_{j}$ of side-length $3$, with
possibly some small part remaining close to the boundary of $B_{\mathbf{i}%
}^{0}$. As the $B_{\mathbf{i}}^{0}$ have side-length $M-2,$ we can place
$\left\lfloor \left(  M-2\right)  /3\right\rfloor ^{d}$ of the $Q$-boxes
without overlaps into $B_{\mathbf{i}}^{0}.$ For a $Q$-box, the probability
that the middle point and all its neighbors belong to $\mathcal{A}$ is
$p^{2d+1}.$ Therefore%
\[
\mathbb{P}\left(  \eta\left(  \mathbf{i}\right)  =0\right)  \leq\left(
1-p^{2d+1}\right)  ^{\left\lfloor \frac{M-2}{3}\right\rfloor ^{d}}.
\]
We choose $M=M\left(  p,d\right)  $ such that%
\begin{equation}
\mathbb{P}\left(  \eta\left(  \mathbf{i}\right)  =0\right)  \leq\frac
{1}{64d^{2}}. \label{fixing_M}%
\end{equation}
For $x\in\mathbb{Z}^{d},$ we write $\mathbf{i}\left(  x\right)  $ for the
index $\mathbf{i}$ such that $x\in B_{\mathbf{i}}$. Remark that $\mathbf{i}%
\left(  0\right)  =0.$ Remark that $M$ depends on $d$ and $p$ only.

Given any self-avoiding nearest-neighbor path connecting $x$ with $0,$ that
is,
\[
\psi=\left(  x_{0}=x,x_{1},\ldots,x_{k}=0\right)
\]
we attach to it a renormalized nearest neighbor path $\overline{\psi}=\left(
\mathbf{i}\left(  x\right)  ,\mathbf{i}_{1},\ldots,\mathbf{i}_{\ell}=0\right)
$ for some $\ell\le k$ in the following way. $\psi$ starts at $x$ which is inside a box
$B_{\mathbf{i}\left(  x\right)  }$. Put $\mathbf{i}_{0}:=\mathbf{i}\left(
x\right)  $. When $\psi$ for the first time leaves $B_{\mathbf{i}_{0}},$ it
enters a box $B_{\mathbf{i}_{1}}$ with $\mathbf{i}_{1}$ being a neighbor of
$\mathbf{i}_{0}$ in $\mathbb{Z}^{d}.$ Then we wait for the next time in which
$\psi$ leaves $B_{\mathbf{i}_{1}}$ and enters a neighbor box $B_{\mathbf{i}%
_{2}}$. This path is not yet self-avoiding, but we can make it so by erasing
successively all the loops. See Figure~\ref{fig:path} for an example.
\begin{figure}[ht!]
\includegraphics[scale=1]{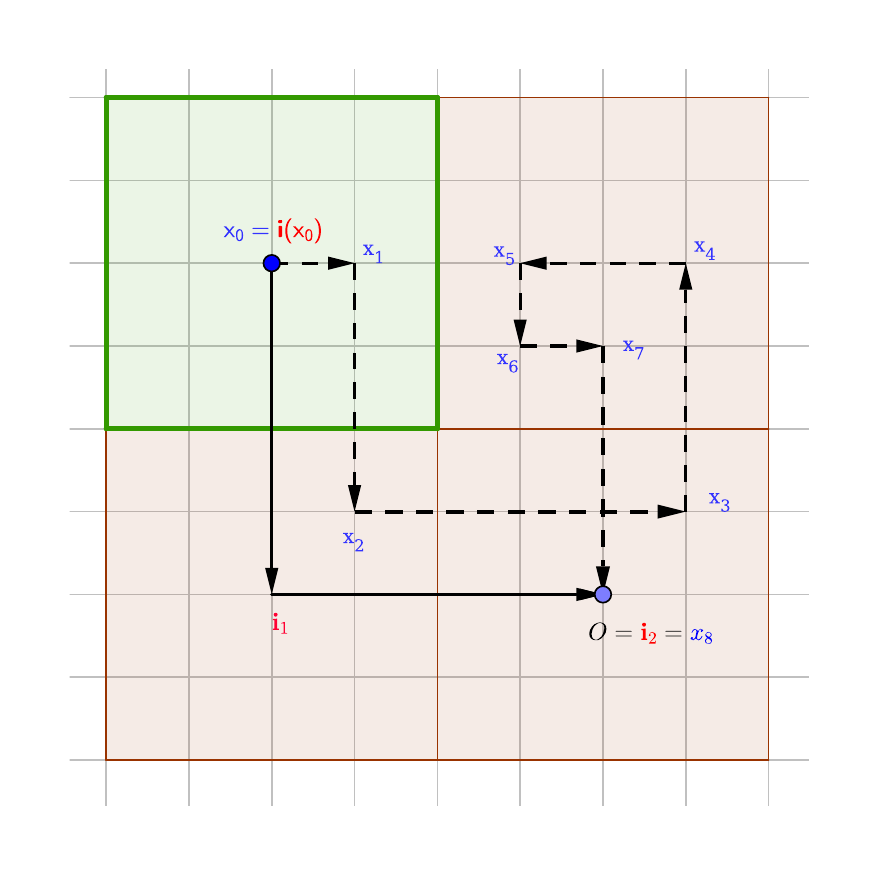}
\caption{The lattice boxes have side length one, while the renormalised $M$-boxes have side length $4$ (one is highlighted in green). The path $\psi=(x_0,\,\ldots,\,x_8)$ is dashed while $\bar\psi=(\mathbf i(x_0),\,\mathbf i_1\,\mathbf i_2)$ is solid.}
\label{fig:path}
\end{figure}

In this way, we proceed and obtain a path from $\mathbf{i}_{0}$ to $0$, which
we indicate as $\mathbf{i}_{0}\rightarrow\mathbf{0}$, of the form $\left(
\mathbf{i}_{0},\mathbf{i}_{1},\ldots,\mathbf{i}_{\ell}=0\right)  .$ Evidently,
we can define an injective mapping $\left\{  0,\ldots,\ell\right\}  \ni
j\mapsto t_{j}\in\left\{  0,\ldots,k\right\}  $ with $x_{t_{j}}\in
B_{\mathbf{i}_{j}}$ (for example letting $t_j$ be the first entrance time of $\psi$ in the box $B_{\mathbf{i}_{j}}$). As this mapping is
injective and the path $\psi$ is self-avoiding, the $x_{t_{j}}$'s are different.

We attach to $\overline{\psi}$ a weight
\[
\overline{\phi}\left(  \overline{\psi}\right)  =\left|  \left\{  j\leq
\ell:\eta\left(  \mathbf{i}_{j}\right)  =1\right\}  \right|
\]
that counts the number of large boxes in which a pinned point lies. From the
construction one obtains that
\[
\overline{\phi}\left(  \overline{\psi}\right)  =\sum_{j=0}^{\ell}%
\mathbf{1}_{\left\{  \eta\left(  \mathbf{i}_{j}\right)  =1 \right\}  }
\leq\sum_{j=0}^{\ell}\mathbf{1}_{\left\{  \eta\left(  \mathbf{i}_{j}\right)
=1 \right\}  } \frac{2+2M^{2d+3}}{2+2\mathrm{d}\left(  x_{t_{j}},\widehat{{A}%
}\right)  ^{2d+3}}%
\]
as $\mathrm{d}\left(  x,\widehat{{A}}\right)  \leq M$ whenever $x\in
B_{\mathbf{i}}$ with $B_{\mathbf{i}}^{0}\cap\widehat{{A}}\neq\emptyset.$
Moreover, recalling \eqref{Def_path_weight}, we can say that
\[
\sum_{j=0}^{\ell}\frac{\mathbf{1}_{\left\{  \eta\left(  \mathbf{i}_{j}\right)
=1 \right\}  } }{2+2\mathrm{d}\left(  x_{t_{j}},\widehat{{A}}\right)  ^{2d+3}%
}\leq\sum_{i=0}^{k}\frac{1}{2+2\mathrm{d}\left(  x_{i},\widehat{{A}}\right)
^{2d+3}}\leq\phi\left(  \psi\right)
\]
and so%
\begin{equation}
\overline{\phi}\left(  \overline{\psi}\right)  \leq\left(  2+2M^{2d+3}\right)
\phi\left(  \psi\right)  . \label{bound_phibar}%
\end{equation}

We have already fixed $M$ above (depending only on $d$ and $p$), and we choose
now $K$ as
\begin{equation}
K:=\left\lceil 20M\left(  2+2M^{2d+3}\right)  \right\rceil . \label{fixing_K}%
\end{equation}

If there exists $x$ with $\mathrm{d}\left(  0,\,x\right)  >Kn$ and
$\widehat{\mathrm{d}}\left(  0,x\right)  \leq10n,$ then there exists a $\psi$
from $x$ to $0$ with $\phi\left(  \psi\right)  \leq10n$, implying by means of
\eqref{bound_phibar} that there exists a path $\overline{\psi}$ from
$\mathbf{i}\left(  x\right)  $ to $0$ with weight%
\[
\overline{\phi}\left(  \overline{\psi}\right)  \leq10\left(  2+2M^{2d+3}%
\right)  n
\]
and%
\[
\mathrm{d}\left(  0,\,\mathbf{i}\right)  >\frac{Kn}{M}.
\]
Setting%
\[
m:=\left\lfloor \frac{Kn}{M}\right\rfloor ,
\]
we see that by our choice (\ref{fixing_K})%
\begin{equation}
\bigcup\limits_{x:\,\mathrm{d}\left(  0,x\right)  >Kn}\left\{  \widehat
{\mathrm{d}}\left(  0,x\right)  \leq10n\right\}  \subset\bigcup
\limits_{\mathbf{i}:\,\mathrm{d}\left(  0,\mathbf{i}\right)  >m}%
\bigcup\limits_{\overline{\psi}:\,\mathbf{i\rightarrow0}}\left\{
\overline{\phi}\left(  \overline{\psi}\right)  \leq\frac{m}{2}\right\}.
\label{Inklusion1}%
\end{equation}
Fix $\mathbf{i}$ with $\mathrm{d}\left(  0,\mathbf{i}\right)  =:l>m$. A path
$\overline{\psi}=\left(  \mathbf{i}_{0}=\mathbf{i,\mathbf{i}}_{1}%
,\ldots,\mathbf{i}_{r}=0\right)  $ has length $r:=\left\vert \overline{\psi
}\right\vert \geq l$, hence
\[
\left\{  \overline{\phi}\left(  \overline{\psi}\right)  \leq\frac{m}%
{2}\right\}  \subset\left\{  \overline{\phi}\left(  \overline{\psi}\right)
\leq\frac{\left\vert \overline{\psi}\right\vert }{2}\right\}  .
\]

The number of paths of length $r\geq l$ on the lattice is bounded by $\left(
2d\right)  ^{r}.$ For every such path $\overline{\psi}$ the $\eta\left(
\mathbf{i}_{j}\right)  $ are i.i.d. with success probability (cf.
\eqref{fixing_M})
\[
\mathbb{P}\left(  \eta\left(  \mathbf{i}\right)  =1\right)  =:\tau\geq
1-\frac{1}{64d^{2}}>\frac{1}{2},
\]
and therefore
\[
\mathbb{P}\left(  \overline{\phi}\left(  \overline{\psi}\right)  \leq\frac
{m}{2}\right)  =\mathbb{P}\left(  \left\{  \overline{\phi}\left(
\overline{\psi}\right)  \leq\frac{m}{2}\right\}  \cap\left\{  \widehat
{\mathcal{A}}\neq\emptyset\right\}  \right)  .
\]
Therefore, for a fixed $\overline{\psi},$ the right-hand side above is bounded
by the probability that a Bernoulli sequence of length $r$ with success
probability $1-\tau$ has at least $r/2$ successes. This probability is bounded
above by (see \cite{ArratiaGordon})
\[
\exp\left[  -rI\left(  \left.  \frac{1}{2}\right\vert 1-\tau\right)  \right]
,
\]
where for $p_{1},p_{2}\in\left(  0,1\right)  $ one defines
\[
I\left(  \left.  p_{1}\right\vert p_{2}\right)  :=p_{1}\log\frac{p_{1}}{p_{2}%
}+\left(  1-p_{1}\right)  \log\frac{1-p_{1}}{1-p_{2}}.
\]
Hence%
\begin{align*}
&  \mathbb{P}\left(  \left\{  \overline{\phi}\left(  \overline{\psi}\right)
\leq\frac{m}{2}\right\}  \cap\left\{  \widehat{\mathcal{A}}\neq\emptyset
\right\}  \right)  \leq\exp\left[  -rI\left(  \left.  \frac{1}{2}\right\vert
1-\tau\right)  \right] \\
&  =\exp\left[  -\frac{r}{2}\left(  \log\frac{1}{2\left(  1-\tau\right)
}+\log\frac{1}{2\tau}\right)  \right]  \leq\left(  2\left(  1-\tau\right)
\right)  ^{r/2}\left(  2\tau\right)  ^{r/2}\leq\left(  4\left(  1-\tau\right)
\right)  ^{r/2}.
\end{align*}
Therefore, for $l>m$,%
\begin{align*}
&  \mathbb{P}\left(  \bigcup\limits_{\overline{\psi}:\mathbf{i\rightarrow0}%
}\left\{  \overline{\phi}\left(  \overline{\psi}\right)  \leq\frac{m}%
{2}\right\}  \cap\left\{  \widehat{\mathcal{A}}\neq\emptyset\right\}  \right)
\leq\sum_{r=l}^{\infty}\left(  2d\right)  ^{r}\left(  4\left(  1-\tau\right)
\right)  ^{r/2}\leq\sum_{r\ge l}\left(  2d\right)  ^{r}\left(  4\frac
{1}{64d^{2}}\right)  ^{r/2}\\
&  =\sum_{r\ge l}\frac{1}{2^{r}}=\frac{1}{2^{l-1}},
\end{align*}
and%
\begin{align*}
\mathbb{P}\left(  \bigcup\limits_{\mathbf{i}:\mathrm{d}\left(  0,\mathbf{i}%
\right)  >m}\bigcup\limits_{\overline{\psi}:\mathbf{i\rightarrow0}}\left\{
\overline{\phi}\left(  \overline{\psi}\right)  \leq\frac{m}{2}\right\}
\cap\left\{  \widehat{\mathcal{A}}\neq\emptyset\right\}  \right)   &  \leq
\sum_{l\ge m+1}\frac{\left\vert \left\{  \mathbf{i}:\mathrm{d}\left(
\mathbf{i},0\right)  =l\right\}  \right\vert }{2^{l-1}}\\
&  \leq c_{5}\left(  d\right)  \sum_{l\ge m+1}\frac{l^{d-1}}{2^{l-1}%
}\leq\left(  \frac{2}{3}\right)  ^{m}%
\end{align*}
for large enough $m.$ Together with (\ref{Inklusion1}), this gives%
\[
\mathbb{P}\left(  \bigcup\limits_{x:\mathrm{d}\left(  0,x\right)  >Kn}\left\{
\widehat{\mathrm{d}}\left(  0,x\right)  \leq10n\right\}  \right)  \leq\left(
\frac{2}{3}\right)  ^{\frac{Kn}{M}}.
\]
This concludes the proof of the lemma choosing $\lambda=\lambda
(p,\,d):=-\left(  K/M\right)  \log(2/3).$
\end{proof}

\subsection{Proof of Theorem \ref{Thm_main}}

We assume now $d\geq5.$ Let $x,\,y\in\mathbb{Z}^{d}$. We have to estimate
$E_{N}^{\varepsilon}\left[  \varphi_{x}\varphi_{y}\right]  $. We may assume
that $x,\,y\in V_{N}$, as otherwise the expression is $0$. It is convenient to
shift everything by $x$:%
\[
E_{N}^{\varepsilon}\left[  \varphi_{x}\varphi_{y}\right]  =E_{V_{N}%
+x}^{\varepsilon}\left[  \varphi_{0}\varphi_{y-x}\right]  =E_{\zeta_{V_{N}%
+x}^{\varepsilon}}\left(  G^{\mathcal{A}}_{V_{N}+x}(0,y-x)\mathbf{1}_{\left\{
0\notin\mathcal{A}\right\}  }\right)
\]
where $\mathcal{A}\subset V_{N}+x$ is distributed according to $\zeta
_{V_{N}+x}^{\varepsilon}$. Substituting $z:=y-x$, we see that we have to
estimate%
\[
\left\vert E_{\zeta_{V_{N}+x}^{\varepsilon}}\left(  G^{\mathcal{A}}_{V_{N}+x}(0,\,z)\mathbf{1}_{\left\{  0\notin\mathcal{A}\right\}  }\right)  \right\vert
\leq E_{\zeta_{V_{N}+x}^{\varepsilon}}\left(  \left\vert G_{\mathcal{A}%
,\,V_{N}+x}(0,\,z)\right\vert \mathbf{1}_{\left\{  0\notin\mathcal{A}\right\}
}\right)  .
\]
Let $\overline{\mathcal{A}}:=\mathcal{A}\cup\left(  V_{N}+x\right)^{\c}.$ For a
fixed realization of $\mathcal{A}$ with $0\notin\mathcal{A},$ $G^{\mathcal{A}}_{V_{N}+x}(0,\cdot)$ is $h_{\overline{\mathcal{A}}}$ restricted to $V_{N}+x.$
Outside this set, $h_{\overline{\mathcal{A}}}$ is of course $0.$

By Proposition \ref{prop:dom}, the distribution of $\overline{\mathcal{A}}$
under $\zeta_{V_{N}+x}^{\varepsilon}$ strongly dominates the Bernoulli law
$\mathbb{P}^{\rho_{-}}$ where $\rho_{-}=\rho_{-}\left(  d,\varepsilon\right)
$ is defined by (\ref{Def_rho}). The Bernoulli domination is proved there
only for the configuration inside $V_{N}+x,$ but as $\overline{\mathcal{A}}$
contains all the points outside $V_{N}+x$, the domination trivially extends to
the measures on $\mathcal{P}\left(  \mathbb{Z}^{d}\right)  $.

Let $K=K\left(  d,\,\varepsilon\right)  $ be as defined in Lemma \ref{Le_main}
with $p$ there equal to $\rho_{-}$. Set%
\[
R_{n}:=\left\{  x\in\mathbb{Z}^{d}:Kn\leq\mathrm{d}\left(  0,x\right)
<K\left(  n+1\right)  \right\}  .
\]
We want to show that we can choose $\delta>0$, depending on $d,\,\varepsilon$ only,
such that%
\begin{equation}
\sup_{N,\,x}{\zeta_{V_{N}+x}^{\varepsilon}}\left(  \sup_{z\in R_{n}}\left\vert
G^{\mathcal{A}}_{V_{N}+x}(0,z)\right\vert \geq\mathrm{e}^{-\delta n}\right)
\leq L\left(  \varepsilon\right)  \mathrm{e}^{-\delta n}. \label{expon_est}%
\end{equation}
Having proved this, Theorem \ref{Thm_main} follows, as $\sup
_{z,\,x,\,N,\,A}\left\vert G^{\mathcal{A}}_{V_{N}+x}(0,z)\right\vert \leq G(0,\,0)<+\infty$ for
$d\geq5$ (cf. \eqref{eq:GtoGamma}) and therefore, if $z\in R_{n}$ for some $n$, by the law of total probability we get
\begin{align*}
\sup_{N,\,x}&\left\vert E_{N}^{\varepsilon}\left[  \varphi_{x}\varphi
_{x+z}\right]  \right\vert \leq\sup_{N,\,x}E_{\zeta_{V_{N}+x}^{\varepsilon}%
}\left(  \left\vert G^{\mathcal{A}}_{V_{N}+x}(0,z)\right\vert \mathbf{1}%
_{\left\{  0\notin\mathcal{A}\right\}  }\right)  \leq L\left(  \varepsilon
\right)  \mathrm{e}^{-\delta n}.
\end{align*}

In order to prove (\ref{expon_est}), set%
\begin{align*}
X_{n}  &  :=\sup_{z\in R_{n}}\left\vert G^{\mathcal{A}}_{V_{N}+x}%
(0,z)\right\vert ,\\
Y_{n}  &  :=\left\Vert G^{\mathcal{A}}_{V_{N}+x}(0,\cdot)\right\Vert
_{\overline{\mathcal{A}},R_{n}},\\
\xi_{n}  &  :=\sqrt{1+\sup\nolimits_{x\in R_{n}}\mathrm{d}(x,\overline
{\mathcal{A}})^{2d+3}}.
\end{align*}
Then%
\eq{}\label{eq:two_vars}
X_{n}\leq\sqrt{\sum_{z\in R_{n}}\left(  G^{\mathcal{A}}_{V_{N}+x}(0,z)\right)
^{2}}\leq\xi_{n}Y_{n}.
\eeq{}
To prove (\ref{expon_est}), we observe that for any $\delta'>0$ and $n\ge n_0(\delta')$ 
\begin{align}
\sup_{N,\,x}\,&{\zeta_{V_{N}+x}^{\varepsilon}}\left(  \xi_n Y_n \geq\mathrm{e}^{-\delta' n}\right)=\sup_{N,\,x}{\zeta_{V_{N}+x}^{\varepsilon}}\left(  \xi_n Y_n \geq\mathrm{e}^{-\delta' n},\,\xi_n< n^{2(d+2)}\right)\nonumber\\
&+\sup_{N,\,x}{\zeta_{V_{N}+x}^{\varepsilon}}\left(  \xi_n Y_n \geq\mathrm{e}^{-\delta' n},\,\xi_n\ge n^{2(d+2)}\right)\nonumber\\
&\le \sup_{N,x}\zeta_{V_{N}+x}^{\varepsilon}\left(  Y_{n}\geq\mathrm{e}%
^{-2\delta' n}\right)+\sup_{N,x}\zeta_{V_{N}+x}^{\varepsilon}\left(  \xi_{n}\geq  n^{2\left(
d+2\right)  }\right).\label{eq:another_two_sum}
\end{align}
Now define $2\delta':=\delta$ where $\delta$ appears in Corollary~\ref{Cor_main}. Notice that
\begin{align}
\zeta_{V_{N}+x}^{\varepsilon}\left(  Y_{n}\geq\mathrm{e}%
^{-\delta n}\right)&=\zeta_{V_{N}+x}^{\varepsilon}\left(  Y_{n}\geq\mathrm{e}%
^{-\delta n},\,R_n\subset C_n^\c\right)+\zeta_{V_{N}+x}^{\varepsilon}\left(  Y_{n}\geq\mathrm{e}%
^{-\delta n},\,R_n\cap C_n\neq\emptyset\right)\nonumber\\
&\le \zeta_{V_{N}+x}^{\varepsilon}\left(  \overline{\mathcal A}=\emptyset\right)+\zeta_{V_{N}+x}^{\varepsilon}\left(  \inf\nolimits_{x:\mathrm{d}\left(
x,0\right)  >Kn}\widehat{\mathrm{d}}_{\overline{\mathcal{A}}}\left(
0,x\right)  \leq{10}n\right).\label{eq:first_term}
\end{align}
In the last inequality we have used Corollary~\ref{Cor_main}. By means of the monotonicity property (\ref{monoton1}) and Bernoulli domination, the right-hand side above is
dominated by%
\[
\mathbb{P}^{\rho^{-}}\left(\overline{\mathcal A}=\emptyset  \right)+ \mathbb{P}^{\rho^{-}}\left(  \inf\nolimits_{x:\mathrm{d}\left(  x,0\right)
>Kn}\widehat{\mathrm{d}}_{\overline{\mathcal{A}}}\left(  0,x\right)  \leq
{10}n\right)  ,
\]
where $\rho^{-}:=\rho^{-}\left(  d,\varepsilon\right)$. With $\lambda$ as of Lemma~\ref{Le_main} we can find $n=n(\lambda)$ large enough such that $\mathbb{P}^{\rho^{-}}\left(\overline{\mathcal A}=\emptyset  \right)\le \exp(-\lambda n)$ applying Lemma~\ref{Le_box_empty}. We plug the result of Lemma~\ref{Le_main} in \eqref{eq:first_term} to get
\eq{}\label{eq:term_first}
\zeta_{V_{N}+x}^{\varepsilon}\left(  Y_{n}\geq\mathrm{e}%
^{-\delta n}\right)\le \e^{-\lambda n}.
\eeq{}
We now look at the second summand of \eqref{eq:another_two_sum}. For large enough $n$ (depending on $d,\,\varepsilon$ only)%
\[
\left\{  \xi_{n}\geq  n^{2\left(  d+2\right)  }\right\}  \subset\left\{\sup_{x\in R_n}
\mathrm{d}(x,\overline{\mathcal{A}})>n^{2}\right\}  .
\]
Using the monotonicity property (\ref{monoton1}), one has%
\[
\sup_{N,x}\zeta_{V_{N}+x}^{\varepsilon}\left(  \sup_{x\in R_n}\mathrm{d}(x,\overline
{\mathcal{A}})>n^{2}\right)  \leq\mathbb{P}^{\rho^{-}}\left(\sup_{x\in R_n}  \mathrm{d}%
(x,\mathcal{A})>n^{2}\right)
\]
which evidently is of order $\exp\left[  -L\times
n^{2}\right]  $ for large $n$. This, \eqref{eq:term_first}, \eqref{eq:first_term} and \eqref{eq:two_vars} prove \eqref{expon_est}.
\bibliographystyle{abbrvnat}
\bibliography{biblio}
\end{document}